\documentclass[11pt,a4paper]{article}
\usepackage[colorlinks=true,linkcolor=blue,citecolor=blue,urlcolor=black,]{hyperref}
\usepackage{physics}
\usepackage{amsmath}
\usepackage{amsthm}
\usepackage{tikz}
\usepackage{eso-pic}
\usepackage{fancyhdr}
\usepackage{amssymb}
\usepackage{fancyhdr}
\usepackage{cite}
\usepackage{tikz-cd}
\usepackage{fontspec}
\usepackage{hyperref}
\usepackage{listings}
\usepackage{geometry}
\usepackage{float}
\usepackage{makecell}

\usepackage[numbers,sort]{natbib}

\numberwithin{equation}{section}
\title{Characterisation of some multivalued harmonic functions on $\R^{n}$}
\author{Yang Li, Parsa Mashayekhi, and Yichen Zhang}
\date{\today}

\newtheorem{thm}{Theorem}[section]
\newtheorem{lem}[thm]{Lemma}

\newtheorem{eg}[thm]{Example}

\newtheorem{cor}[thm]{Corollary}

\newtheorem{prop}[thm]{Proposition}
\theoremstyle{definition}
\newtheorem{rmk}[thm]{Remark}

\newtheorem{Def}[thm]{Definition}

\newcommand{\ie}{\emph{i.e.} }
\newcommand{\cf}{\emph{cf.} }

\newcommand{\R}{\mathbb{R}}

\newcommand{\bmh}{\boldsymbol{h}}
\newcommand{\bma}{\boldsymbol{a}}

\newcommand{\C}{\mathbb{C}}
\newcommand{\Z}{\mathbb{Z}}

\begin{document}
\maketitle
\begin{abstract}
Using different approaches, Simon Donaldson and Dashen Yan recently constructed $\mathbb Z_2$-harmonic functions on $\mathbb R^3$ and on $\mathbb R^n$ for $n\geqslant 3$, respectively. Their examples have quadratic growth at infinity and $\mathcal{O}(r^{3/2})$ local growth near a smooth codimension-two branching set $\Sigma$. We show that these properties uniquely characterise such $\mathbb Z_2$-harmonic functions up to rigid motions of $\mathbb R^n$.
\end{abstract}

\section{Background and Introduction}
    Multivalued harmonic functions (resp. $1$-forms, spinors) appear ubiquitously in modern differential geometry. For instance,

\begin{enumerate}
\item In Almgren's regularity theory, multivalued harmonic functions appear as graphical approximations to area minimising currents converging to some limiting current with multiplicity.

    \item  Taubes \cite{Taubes-PSL2C,Taubes-ZeroLoci} and Haydys--Walpuski \cite{haydys2015compactness} studied  non-compactness phenomena for gauge-theoretic equations in three and four dimensions, such as flat $PSL(2,\C)$ connections, the Vafa-Witten equation, and generalised Seiberg-Witten equations with multiple spinors, where $\Z_2$-harmonic 1-forms/spinors arise as blow up limits.

    \item  In calibrated geometry, Siqi He \cite{He-BranchedSLag} recently constructed branched deformations of special Lagrangians, starting from $\Z_2$-valued harmonic 1-forms with smooth branching locus, satisfying some non-degeneracy conditions.

    \item  In Donaldson's programme \cite{Donaldson-Adiabatic} for adiabatic limits of $G_2$-manifolds collapsing along a coassociative K3 fibration, the limiting data is encoded by a maximal submanifold with local branching behaviour along some links in $S^3$, and $\Z_2$-harmonic functions can be viewed as local linearised models for these maximal submanifolds.
\end{enumerate}

\begin{rmk}
There has been substantial progress on $\Z_2$ harmonic functions/1-forms/spinors.
\begin{itemize}
    \item The \emph{deformation theory} was developed by Donaldson \cite{Donaldson-Deformation} for $\Z_2$-harmonic functions and $1$-forms in general dimensions, and by Takahashi \cite{Takahashi-Moduli} and Parker \cite{Parker-Deformation}  for $\Z_2$-harmonic spinors in dimension $3$ with embedded curves as branching sets, using different approaches. Related index theory for embedded curves and graphs was developed by Takahashi \cite{Takahashi-Index} and Haydys--Mazzeo--Takahashi \cite{Haydys-Mazzeo-Takahashi-IndexGraph}.

    \item (\emph{Existence}) Doan--Walpuski \cite{Doan-Walpuski} proved the existence of $\Z_2$-harmonic spinors by studying wall-crossing phenomena for the Seiberg--Witten equation in dimension $3$. He--Parker \cite{He-Parker} and Salm \cite{Salm} established gluing results for $\Z_2$-harmonic $1$-forms and spinors. Various examples were constructed by Taubes--Wu \cite{Taubes-Wu}, Weifeng Sun \cite{Sun-Pell}, Haydys--Mazzeo--Takahashi \cite{Haydys-Mazzeo-Takahashi-Examples}, Donaldson \cite{Donaldson-Twistor} and Dashen Yan \cite{Yan-Construction}. 
    \item The \emph{deformation rigidity} of the Taubes-Wu homogeneous solutions on $\R^3$  was obtained by Chen--He \cite{Chen-He} and more recently generalised by Haydys--He--Salm \cite{Haydys-He-Salm}. 
\end{itemize}
\end{rmk}

In this paper we study $\Z_2$-harmonic functions $(u, \Sigma)$ on $\R^n$, where $\Sigma\subset \R^n$ is a real codimension-two branching set, and $u$ is a harmonic section of a flat real line bundle $\mathcal{L}$ over $\R^n\setminus\Sigma$, with $\Z_2$-monodromy around any small loop linking $\Sigma$. Locally around $\Sigma$, the function $u$ has an asymptotic expansion
\[
 u(t,z)=\operatorname{Re}({\bf{A}}(t)z^{1/2})+\operatorname{Re}({\bf{B}}(t)z^{3/2})+\mathcal{O}(|z|^{2+\alpha})
\]
where ${\bf{A}}\in \Gamma(N_\Sigma^{-1/2}),{\bf{B}}\in \Gamma(N_\Sigma^{-3/2})$ are smooth sections of certain bundles over $\Sigma$, and $z$ is the complex coordinate normal to $\Sigma$, for some $0<\alpha<1/2$. However, for certain applications to special Lagrangians and branched maximal submanifolds, one naturally needs to impose the condition that the coefficient function $\bf{A}$ vanishes along the branching locus $\Sigma$, or equivalently $u=\mathcal{O}(r^{3/2})$ for $r=\text{dist}(\cdot, \Sigma)$. Thus $\Sigma$ should be viewed as a nonlinear free boundary that needs to be solved along with $u$, rather than prescribed \textit{a priori}.

Recently, Donaldson \cite{Donaldson-Twistor} constructed certain explicit solutions on $\R^3$, while Yan \cite{Yan-Construction} constructed such solutions on $\R^n$ for all $n\geqslant 3$, using different methods.

\begin{thm}
Let $n\geqslant 3$. 
There exist $\Z_2$-harmonic functions $u$ on $\R^n$, with branching sets $\Sigma$ being ellipsoids, such that $u=\mathcal{O}(\textnormal{dist}(\cdot, \Sigma)^{3/2})$ along $\Sigma$. 
These $u$ are asymptotic at infinity to non-degenerate trace-free quadratic forms of index $1$ or $n-1$. 
\end{thm}

\begin{rmk}
As observed by Dashen Yan \cite{Yan-Construction}, such $\Z_2$-harmonic functions can  be extracted as the limits for families of Lawlor necks \cite{Lawlor} with small characterising angles, double branched over $\R^n$ with ellipsoids as branching sets. 
\end{rmk}

Our goal is to prove the global uniqueness result. 

\begin{thm}\label{mainthm}
Let $n\geqslant 3$. Suppose $u$ is a $\Z_2$-harmonic function on $\R^n$ satisfying the following conditions:
\begin{itemize}
    \item The branch locus $\Sigma\subset |du|^{-1}(0)$ is a non-empty smoothly embedded compact codimension--$2$ submanifold; 
    \item Around $\Sigma$, the coefficient ${\bf{A}}\in \Gamma(N_\Sigma^{-1/2})$ in the leading asymptotic expansion of $u$ is identically zero;
    \item At the infinity of $\R^n$, 
    \begin{equation}
        \label{eqn:asym-of-u}
 u \sim \pm V + \mathcal{O}(|x|^\mu),
    \end{equation}
where $V$ is a non-degenerate homogeneous degree $2$ polynomial and the asymptotic rate satisfies $\mu<2$. 
\end{itemize}

Then, up to rigid motions of $\R^n$, $u$ agrees with the model $\Z_2$-harmonic function $g_{\bmh,\delta}$ constructed by Donaldson and Yan (see below Example \textnormal{\ref{eg:z2}}).
\end{thm}

\begin{rmk}
Weifeng Sun \cite{Sun-Pell} and Haydys--Mazzeo--Takahashi \cite{Haydys-Mazzeo-Takahashi-Examples} showed that, for any \textit{fixed} $\Sigma$ and any harmonic polynomial on $\R^n$, there exists a unique $\Z_2$-harmonic function in the general class with the prescribed asymptote at infinity. This implies in particular that once the ellipsoid is fixed as the branching set, the $\Z_2$-harmonic functions found by Donaldson and Dashen Yan are uniquely determined by the quadratic asymptote at infinity. In contrast, our result does not assume $\Sigma$ is known \textit{a priori}.
\end{rmk}

\subsection{Proof Strategy}

The proof of Theorem \ref{mainthm} consists of two parts.
\begin{itemize}
    \item \textbf{Analytical Preparations.} We adapt Almgren's frequency function from the study of $Q$-valued harmonic functions (\cf \cite{DeLellis-Spadaro}) and prove a monotonicity formula for the $\Z_2$-harmonic $1$-form $du$, relating the vanishing order around the singular set to the polynomial growth rate at infinity. This forces $du$ to be non-degenerate and to have no zero away from the branching set. Thus we can build an approximately special Lagrangian graph
    $L_\delta:=\operatorname{graph}(\delta du)$, which is 
$C^{1,\alpha}$-regular and embedded.

    \item \textbf{Floer-theoretic part.}
    Our strategy can be viewed as an adaptation of the theorem of Imagi--Joyce--Oliveira dos Santos \cite{IJS} on the classification of special Lagrangians $L$ in $\C^n$ asymptotic 
    to the union of two planes $\Pi_0\cup \Pi_\phi$, 
    \[
    \Pi_0=\R^n,\quad \Pi_\phi= (e^{i\phi_1},\ldots e^{i\phi_n})\R^n,\quad 0<\phi_k<\pi,\quad  k=1,\ldots n.
    \]
    They first partially compactify $\C^n$ into a Liouville manifold $M$, by adding two copies of $\R^n$ at infinity, and compactify the special Lagrangian into an exact, graded compact Lagrangian $\bar{L}$. By analysing the intersection pattern with $\Pi_0, \Pi_\phi$ and the $\R^n$ at infinity, and appealing to the classification of objects in the derived Fukaya category $D^bFuk$, they can show that $\sum\limits_{k=1}^n \phi_k=\pi$ or $(n-1)\pi$, and the compactified Lagrangian agrees with the compactified Lawlor neck in $D^bFuk$. Furthermore, by considering the Lagrangian potential, they can define an analytic invariant $A(L)>0$, which reflects the scale parameter in the family of Lawlor neck models.

Imagi--Joyce--Oliveira dos Santos then prove that $L$ coincides with the appropriate choice of Lawlor neck special Lagrangian, by a version of the Thomas-Yau argument. The core idea is to remove the Lagrangian intersection points of Floer degree $0,n$ inside $\C^n$, which they accomplish by a delicate argument using real analyticity. Floer theory then produces a holomorphic strip with boundary on the Lawlor neck and (a small Hamiltonian perturbation of) the unknown $\bar L$, with corners at the Lagrangian intersection points of degree $0,n$, which necessarily lie at infinity. However, the topological energy formula for holomorphic curves then implies that its area is zero, which is a contradiction unless $L$ coincides with the Lawlor neck.

In our case, given the $\Z_2$ harmonic function $u$, the graph of $\delta du$ for $0<\delta\ll 1$ can be viewed as an approximate special Lagrangian, and we wish to compare this to the model solution $g_{\bmh,\delta}$. We follow Imagi--Joyce--Oliveira dos Santos to introduce the partial compactification, which now \textit{depends on the small parameter $\delta>0$}. As a limiting analogue of the fact that
\[
\sum_{k=1}^n \phi_k= \pi~ \textit{ or}~  (n-1)\pi, \quad A(L)>0, 
\]
we will show that up to translation and rotation, the asymptotic harmonic polynomial at infinity is
\[
\sum_{k=1}^n M_k x_k^2- M_0 ,
\]
where the quadratic part has Morse index $1$ or $n-1$, and $M_0>0$ (resp. $M_0<0$). By adapting the real analyticity argument of Imagi--Joyce--Oliveira dos Santos, and using the harmonic equation instead of the special Lagrangian graph equation, we can remove the Lagrangian intersection points of degree $0,n$, away from a small neighbourhood of the branching locus (resp. the infinity). Near the branching locus (resp. the infinity), we may still have Lagrangian intersections, but we gain some more control on the smallness of the Lagrangian potential. The same Floer theoretic ingredients then produce a \emph{holomorphic curve with small area} of order $o(\delta)$. If the graph of $\delta du$ does not coincide with the model, then an extra monotonicity formula argument gives the reverse \emph{area lower bound}
$\gtrsim \delta$, which is a contradiction.

\end{itemize}

\begin{rmk}[Alternative strategy] \label{rmk:alternativestrategy}
  Given the analytical step, it is conceivable to develop a \textit{non-compact} version of Siqi He's result \cite{He-BranchedSLag}, namely to first perturb the approximate special Lagrangian graph $L_\delta:=\operatorname{graph}(\delta du)$ to a genuine special Lagrangian, and then apply Imagi--Joyce--Oliveira dos Santos \cite{IJS} to deduce uniqueness. A possible technical difficulty is that the two asymptotic planes of $L_\delta$ are \textit{only approximately special Lagrangian} when $\delta \ll 1$. Therefore, besides deforming the branch set $\Sigma$ and dealing with a singular Riemannian metric as in \cite{He-BranchedSLag}, one would also need to correct the ansatz at infinity without introducing large errors, or else work in weighted spaces of multivalued functions allowing growth at infinity. Our treatment instead avoids solving the nonlinear special Lagrangian equation.
\end{rmk}

\textbf{Organisation}. In Section \ref{Sec:preliminaries}, we review the explicit constructions of the Lawlor necks and $\Z_2$-harmonic functions constructed by Donaldson and Yan. In Section \ref{sec:uniquen=2}, we derive a uniqueness result for the case $n=2$ using complex analysis. In Section \ref{Sec:analyticalpreparations} and Section \ref{Sec:floer-theory}, we implement the proof strategy discussed above. At the end, we list two related questions. 

\textbf{Acknowledgements}. Y. L. is supported by the Royal Society URF. P.M.  is supported by the Engineering and Physical Sciences Research Council [EP/S021590/1], the EPSRC Center for Doctoral Training in Geometry and Number Theory (The London School of Geometry and Number Theory), University College London. P.M is also supported by King’s College London during his PhD. Y. Z. is supported by the London School of Geometry and Number Theory Centre for Doctoral Training (CDT), a joint venture between University College London, Imperial College London and King’s College London.

We thank Professor Simon Donaldson for discussions, encouragements and inspirations, and  Richard Thomas and Aleksander Doan for their interest in this work.

\section{Preliminaries}
\label{Sec:preliminaries}
We start by reviewing Lawlor necks and corresponding $\Z_2$-harmonic functions. 

\subsection{Data Matching}

Let
\begin{equation}
    \label{eqn:Cn-CYstr}
    (\omega,\Omega,\lambda)=
    \left(
    \frac{\sqrt{-1}}{2}\sum\limits_{k=1}^n dz_k\wedge d\bar z_k,\,
    dz_1\wedge dz_2\wedge \cdots\wedge dz_n,\,
    {\sum_{k=1}^{n}}\left(\frac{1}{2}x_kdy_k-\frac{1}{2}y_kdx_k\right)
    \right)
\end{equation}
be the standard Calabi--Yau structure on $\C^n$, with a \textit{fixed} choice of Liouville form $\lambda$, where $z_k=x_k+\sqrt{-1}y_k$. For an exact Lagrangian $L\subset \C^n$, we denote its Lagrangian potential by $f_L$, \ie
\[
    df_{L}=\lambda|_{L}.
\]

We first recall the Lawlor necks, an $n$-dimensional family of special Lagrangians in $\C^n$, which admit explicit descriptions \cite{Lawlor,harvey1990spinors}. 
\begin{eg}[Lawlor necks] \label{eg:Lawlor}
    The Lawlor necks $L_{\phi,A}$ are closed embedded exact special Lagrangians in $\C^n$, asymptotic at infinity to the union of two transverse special Lagrangian planes
    \[
        \Pi_0=\R^n,\quad
        \Pi_\phi=(e^{i\phi_1},\cdots,e^{i\phi_n})\R^n,\quad
        0<\phi_k<\pi,\quad
        \sum_{k=1}^n\phi_k=\pi.
    \]

    From the viewpoint of symplectic topology, they can be regarded as realizations of the Lagrangian handle $S^{n-1}\times \R$ appearing in the Lagrangian connected sum construction. This motivates the following ansatz, under which the special Lagrangian condition becomes an ODE system for $z_1(s),\cdots,z_n(s)$ admitting explicit solutions \eqref{eqn:lawlorneckzk}:
    \begin{equation}
        \label{eqn:Lawlorneckansatz}
        L_{\phi,A}:=
        \{(z_1(s)w_1,\cdots,z_n(s)w_n):s\in \R,(w_1,\cdots,w_n)\in S^{n-1}\}.
    \end{equation}
    Let $n>2$. We define the polynomials $p_{\bma}$ and $P_{\bma}$ by
    \[
        p_{\bma}(x)=(1+a_1x^2)\cdots (1+a_nx^2),
        \quad
        P_{\bma}(x)=\frac{p_{\bma}(x)-1}{x^2}.
    \]
    We then define the angle parameters $\phi=(\phi_1,\cdots,\phi_n)$ and the scale parameter $A>0$ of the Lawlor neck $L_{\phi,A}$ by
    \begin{equation}
        \label{eqn:def-lawlor-data}
        \phi_k=a_k\int_{\R}\frac{dx}{(1+a_kx^2)\sqrt{{P_{\bma}(x)}}},
        \quad
        A=\int_\R \frac{dx}{2\sqrt{P_{\bma}(x)}}.
    \end{equation}
    It is easy to check that $\sum\limits_{k=1}^n\phi_k=\pi$. Now, setting
    \begin{equation}
        \label{eqn:lawlorneckzk}
        z_k(s)=e^{i\psi_k(s)}\sqrt{a_k^{-1}+s^2},
        \quad
        \textit{where}\quad
        \psi_k(s):=\int_{-\infty}^s
        \frac{a_kdx}{(1+a_kx^2)\sqrt{P_{{\bma}}(x)}},
    \end{equation}
    gives the desired solution to the ansatz \eqref{eqn:Lawlorneckansatz}.
\end{eg}

We now align the $\Z_2$-harmonic functions  constructed by Donaldson \cite{Donaldson-Twistor} and Yan \cite{Yan-Construction} with Lawlor necks. 

\begin{eg}[$\Z_2$-harmonic functions with quadratic growth, \textup{\cite{Yan-Construction}}]\label{eg:z2}
\label{prop:Z2funcs}
    Let $\bmh=(h_1,\cdots,h_{n-1})>0$ and $\delta>0$. We define the polynomials $q_{\bmh}$ and $Q_{\bmh}$ by
    \[
        q_{\bmh}(x)=(1+h_1^{-2}x^2)\cdots (1+h_{n-1}^{-2}x^2),
        \quad
        Q_{\bmh}(x)=\frac{q_{\bmh}(x)-1}{x^2}.
    \]
    We also introduce the parameters $(b_0(\bmh),\cdots,b_n(\bmh))$ by
    \begin{equation}\label{eqn:Z2-harmonic-data}
        b_0(\bmh)=\frac{1}{4}\int_{\R}\frac{dx}{\sqrt{q_{\bmh}(x)}}>0,
        \quad
        b_k(\bmh)=\frac{1}{4}\int_{\R}\frac{dx}{(h_k^2+x^2)\sqrt{q_{\bmh}(x)}},
        \quad
        b_n(\bmh)=-\sum_{k=1}^{n-1}b_k(\bmh).
    \end{equation}
    Then, for any $\delta>0$, there exists a non-degenerate $\Z_2$-harmonic function
    \[
        g_{\bmh,\delta}=\delta g_{\bmh}
    \]
    on $\R^n$, with branching locus
    \[
        E_{\bmh}
        =
        \left\{
        \sum\limits_{k=1}^{n-1}\frac{x_k^2}{h_k^2}=1,\,
        x_n=0
        \right\},
    \]
    and with asymptotic expansion at infinity
    \begin{equation}
        \label{eqn:Z2-asym-infty}
        g_{\bmh}
        \sim
        \pm\left(
        -b_0(\bmh)+\sum_{k=1}^n b_k(\bmh)x_k^2+O(|x|^{2-n})
        \right).
    \end{equation}

    The multivalued graph $L_{\delta}'=\operatorname{graph}(\delta dg_{\bmh})$ admits an explicit description similar to \eqref{eqn:Lawlorneckansatz}:
    \[
        L_\delta'
        =
        \left\{
        (x,\delta \xi)\in \C^n:
        x_k=w_k\sqrt{h_k^2+s^2},\,
        \xi_k=w_k\beta_k(s)\sqrt{h_k^2+s^2},\,
        x_n=-sw_n,\,
        \xi_n=w_n(1-s \beta_n(s))
        \right\},
    \]
    where $(x,\xi)\in T^*\R^n\cong \C^n$, $s\in \R$, $w=(w_1,\cdots,w_n)\in S^{n-1}$, $1\leqslant k\leqslant n-1$, and
    \[
        \beta_k(s)=\int_0^s\frac{dx}{(h_k^2+x^2)\sqrt{q_{\bmh}(x)}},
        \quad
        \beta_n(s)=-\int_0^s\frac{ Q_{\bmh}(x)dx}{q_{\bmh}(x)+\sqrt{q_{\bmh}(x)}}.
    \]
    In particular, $L_\delta'$ is a smooth closed embedded Lagrangian in $T^*\R^n=\C^n$. 
\end{eg}
For precise definition of the multivalued graph $\operatorname{graph}(\delta dg_{\bmh})$, we refer to Section \ref{sec:multivaluedgraph}. 
\begin{rmk}\label{rmk:scaling}
We summarise several scaling properties of the Lawlor necks and the above $\Z_2$-harmonic functions. 
    \begin{enumerate}
        \item[\textnormal{(1)}] For the Lawlor necks $L_{\phi,A}$, under the coordinate rescaling $z\mapsto \lambda z$ in $\C^n$, we have
        \[
            \bma \mapsto \lambda^{-2}\bma,
            \quad
            \phi\mapsto \phi,
            \quad
            A\mapsto \lambda^2A.
        \]
        Heuristically, the scale parameter $A$ corresponds to the symplectic area of a holomorphic triangle formed by $\Pi_0$, $L_{\phi,A}$ and $\Pi_\phi$.
        On the side of $\Z_2$-harmonic functions, this corresponds to the rescaling $(\bmh,\delta)\mapsto (\lambda h,\lambda \delta)$, and the constant term $\delta M_0$ has an analogous meaning for the approximate Lagrangian graphs $L_\delta'=\operatorname{graph}(\delta dg_{\bmh})$ (\cf Corollary \ref{cor:M0>0}).  
        
        \item[\textnormal{(2)}] For the $\Z_2$-harmonic functions $g_{\bmh,\delta}$, under the coordinate rescaling $x\mapsto \lambda x$ in $\R^n$, we have
        \[
            E_{\bmh}\mapsto E_{\lambda\bmh},
            \quad
            b_0(\bmh)\mapsto \lambda b_0(\bmh),
            \quad
            b_k(\bmh)\mapsto \lambda^{-1}b_k(\bmh),
            \quad
            1\leqslant k\leqslant n,
        \]
        and hence
        \[
            g_{\lambda \bmh,\delta}(x)
            =
            \lambda g_{\bmh,\delta}(x/\lambda)
            \sim
            \pm\left(
            -\lambda\delta b_0(\bmh)
            +
            \delta\lambda^{-1}\sum_{k=1}^{n}b_k(\bmh)x_k^2
            +
            \mathcal{O}(\lambda^{n-1}\delta |x|^{2-n})
            \right).
        \]

        \item[\textnormal{(3)}] For the $\Z_2$-harmonic functions $g_{\bmh,\delta}$, if $\bmh$ is fixed, multiplication by a number $\delta\ll 1$ gives a family of \textit{approximate} special Lagrangian multivalued graphs
        $L_{\delta}':=\operatorname{graph}(\delta dg_{\bmh})$
        for $\delta \ll 1$ (\cf Proposition \ref{prop:appslag}), which are double branched covers over $\R^n$ with branching set $E_{\bmh}$. This reflects the fact that non-degenerate $\mathbb Z_2$-harmonic $1$-forms describe branched deformations of the special Lagrangian $\R^n\subset T^*\R^n=\C^n$.

        \item[\textnormal{(4)}] For the Lawlor necks $L_{\phi,A}$, the $U(n)$-rotation
            $U=\operatorname{diag}(e^{i(\pi-\phi_1)},\cdots,e^{i(\pi-\phi_n)})$
        maps $\Pi_0,\Pi_{\phi}$ to $\Pi_{\pi-\phi},\Pi_0$ respectively, where
        $\pi-\phi:=(\pi-\phi_1,\cdots,\pi-\phi_n)$.
        For each $A>0$, we then define
        $L_{\pi-\phi,-A}:=UL_{\phi,A}$,
        with $\sum\limits_{k=1}^n(\pi-\phi_k)=(n-1)\pi$, asymptotic to $\Pi_{\pi-\phi}\cup \Pi_0$ at infinity. On the side of $\Z_2$-harmonic functions, this corresponds to switching the two branches of the multivalued graph $L_{\delta}'$ (\cf Corollary \ref{cor:indexv}). 
    \end{enumerate}
\end{rmk}

\begin{prop}[Parameter Matching]
\label{prop: para-match}
    We have the following parameter matching maps.
\begin{enumerate}
    \item[\textnormal{(1)}] \textnormal{(Lawlor necks: $\bma$ to $(\phi,A)$)} There exists a one-to-one correspondence $\Phi_{a}$ between the following sets:
    \[
        \Phi_{a}:\R_{>0}^n
        \to
        \left\{
        (\phi,A):
        \sum_{k=1}^{n}\phi_k=\pi,\,
        \phi_k>0,\,
        A>0
        \right\},
        \quad
        \bma \mapsto (\phi,A),
    \]
    defined by \eqref{eqn:def-lawlor-data}.

    \item[\textnormal{(2)}] \textnormal{($\Z_2$-harmonic functions: $(\bmh,\delta)$ to the asymptote at infinity)}
    There exists a one-to-one correspondence $\Phi_{b}$ between the following sets:
    \[
        \Phi_{b}:
        \R^{n-1}_{>0}\times \R_{>0}
        \to
        \left\{
        (b_0,b_1,\cdots,b_{n-1},b_n):
        b_0,\cdots,b_{n-1}>0,\,
        b_n=-\sum_{k=1}^{n-1}b_k
        \right\},
    \]
    \[
        (\bmh,\delta)\mapsto \delta(b_0(\bmh),\cdots,b_n(\bmh)),
    \]
    defined by \eqref{eqn:Z2-harmonic-data}.
\end{enumerate}
\end{prop}
    \begin{proof}
        The first item is well-known (\cf \cite[Example 3.9]{IJS}). The second item follows from the one-to-one correspondence between $\bmh=(h_1,\cdots,h_{n-1})$ and $(b_1,\cdots,b_{n}(\bmh))$ proved by Yan and Zhou (\cf \cite[Corollary 4.3, Remark 4.4]{Yan-Construction}) and the scaling property (2) in Remark \ref{rmk:scaling}.  
    \end{proof}
\begin{rmk}
   We take this opportunity to sketch how the $\Z_{2}$-harmonic function is extracted from the Lawlor neck family, corresponding to the converse direction of item (3) in Remark \ref{rmk:scaling}. Fix $\bmh$, and then
   \[
        \bma=(h_1^{-2},\cdots,h_{n-1}^{-2},\delta^{-2})
   \]
   singles out a family of Lawlor necks $L_{\phi,A,\delta}$ parametrised by $\delta$. After the following rotation, all fibre coordinates will be of order $\mathcal{O}(\delta)$ over a fixed large ball in $\Pi_0$.
   
   Consider the $\delta$-dependent $SU(n)$ rotations $
        U_{\delta}
        =
        \operatorname{diag}
        (e^{-i\phi_1/2},\cdots,e^{-i\phi_{n-1}/2},e^{i(\pi-\phi_n)/2})$, the projection $U_\delta L_{\phi,A,\delta}\to \Pi_0$ defines a double branched cover with branching set $E_{\bmh}$. Compressing the cotangent fibre coordinates (\ie taking $\delta \to 0^+$), the $\Z_2$-harmonic function $g_{\bmh}$ arises as the rescaled limit of the special Lagrangian potential for this family, with respect to \textit{another} Liouville form
   $\lambda=-\sum\limits_{k=1}^ny_kdx_k$. The translation formula is then given by \eqref{eqn:Lagpotential}.
    
   A caveat is that the asymptotic planes
   $\Pi^\delta_\pm=(e^{\pm i\psi_1},\cdots,e^{\pm i\psi_n})\R^n$
   of $L_\delta'=\operatorname{graph}(\delta dg_{\bmh})$ are \textit{only} approximately special Lagrangian planes, with angle error $\mathcal{O}(\delta^3)$, where
   \[
        \psi_k=\operatorname{arctan}(2\delta b_k(\bmh)),
        \quad
        \left|\sum_{k=1}^{n} \psi_k\right|=\mathcal{O}(\delta^3),
        \quad
        \textit{for}~\delta\ll 1.
   \]
\end{rmk}

\subsection{Invariant Setup and Analysis Foundations}
We recall the invariant setting and local structure theory of $\Z_2$-harmonic functions (resp. $1$-forms) developed by Donaldson \cite{Donaldson-Deformation}.

Let $\Sigma \subset \R^n$ be a co-oriented compact embedded codimension-$2$ submanifold, and let $\mathcal{L}$ be the unique flat real line bundle over $\R^n\setminus \Sigma$ with holonomy $-1$ around any small loop linking $\Sigma$, corresponding to a representation
$\chi:\pi_1(\R^n\setminus \Sigma)\to \{\pm 1\}.$
Harmonic sections of $\mathcal{L}$ are called multivalued ($\Z_2$-) harmonic functions on $\R^n$. We define adapted H\"older spaces with a fixed H\"older exponent $\alpha\in (0,\frac{1}{2})$ as follows. 

Consider the model case $\mathbb R^{n-2}_t\times \C_z$ with branching set $\R^{n-2}_t\times\{0\}$. For any $s\in \Gamma(\mathcal{L})$, we define
\[
    \|s\|_{C^{0,\alpha}}
    :=
    \sup_{|p-p'|\leqslant \frac{1}{2}\min\{|z|,|z'|\}}
    \frac{|s(p)-s(p')|}{|p-p'|^\alpha},
\]
for $p=(z,t)$ and $p'=(z',t')$, where the difference between $s(p)$ and $s(p')$ is measured by parallel transport along the segment from $p$ to $p'$. As a caveat, unlike the single-valued case, the $\Z_2$-monodromy implies that
$$\|s\|_{C^0}\leqslant C\|s\|_{C^{0,\alpha}},$$
by considering parallel transport around a polygon.

Let $\mathcal{T}_k$ be the set of differential operators given by monomials of degree $k$ in the $n$ commuting vector fields $\partial_{t_i},r\partial_r,\partial_\theta$.
We define the $\mathcal{D}^{k,\alpha}$-norm of a section $s$ by
\[
    \|s\|_{\mathcal{D}^{k,\alpha}}
    :=
    \max_{0\leqslant j\leqslant k,D\in \mathcal{T}_j}
    \|Ds\|_{C^{0,\alpha}}.
\]

For a compact embedded smooth $\Sigma$, the above adapted H\"older spaces are defined near $\Sigma$ in the standard way, denoted as $C^{0,\alpha}_{\operatorname{loc}}$ and $\mathcal{D}^{k,\alpha}_{\operatorname{loc}}$, for instance by covering $\Sigma$ with finitely many Fermi coordinate balls. We write $C^{k,\alpha}_{\operatorname{loc}}$ for the usual H\"older space. For convenience, we also define $s\in \mathcal{D}_{\operatorname{loc}}^{\infty,\alpha}$ if $s\in \mathcal{D}^{k,\alpha}_{\operatorname{loc}}$ for all $k\geqslant 0$. 

The following structure theorem is due to Donaldson.
\begin{thm}[\cite{Donaldson-Deformation}]\label{thm:structureofz2}
    Let $u\in \Gamma(\mathcal{L})$ be a harmonic section in $W^{1,2}_{\operatorname{loc}}(\R^n)$. Then $u\in \mathcal{D}^{\infty,\alpha}_{\operatorname{loc}}$. For any $p\in \Sigma$, let $(t,z)=(t_1,\cdots,t_{n-2},z)$ be a Fermi coordinate system around $p=(0,0)$, where $z$ is the complex coordinate transverse to $\Sigma$. Then
    \begin{equation}
        \label{eqn:asymptotestructuretheory}
        u(t,z)=\operatorname{Re}({\bf{A}}(t)z^{1/2})+\operatorname{Re}({\bf{B}}(t)z^{3/2})+E,
    \end{equation}
    where the coefficients ${\bf{A}}(t)\in \Gamma(N_{\Sigma}^{-1/2})$ and ${\bf{B}}(t)\in \Gamma(N_\Sigma^{-3/2})$ are smooth sections of certain powers of the normal bundle $N_\Sigma$, and the error $E\in C^{2,\alpha}_{\operatorname{loc}}\cap \mathcal{D}^{\infty,\alpha}_{\operatorname{loc}}$ satisfies $|\nabla^k E|\leqslant C(u,\alpha)|z|^{2+\alpha-k}$ for $k=0,1,2$, together with the corresponding H\"older norm bound. 
\end{thm}

\begin{Def}
    We call $u$ (respectively, $du$) a non-degenerate $\Z_2$-harmonic function (respectively, $1$-form) if the leading coefficient ${\bf{A}}(t)\in \Gamma(N_{\Sigma}^{-1/2})$ vanishes identically, and ${\bf{B}}(t)\in \Gamma(N_\Sigma^{-3/2})$ is nowhere vanishing. 
\end{Def}

\subsection{Multivalued graphs}
\label{sec:multivaluedgraph}
Let $u$ be a non-degenerate $\Z_2$-harmonic function on $\R^n$. The representation
$\chi:\pi_1(\R^n\setminus \Sigma)\to \{\pm1\}$
determines an associated branched double cover $p:\widetilde{\R^n}\to \R^n$ with branching locus $\Sigma$. The space $\widetilde{\R^n}$ is connected by the monodromy condition. Near $\Sigma\subset \widetilde{\R^n}$, the covering map $p$ is modelled on
\[
    \R^{n-2}\times \C \to \R^{n-2}\times \C,
    \qquad
    (t,w)\mapsto (t,z):=(t,w^2),
\]
identifying $\Sigma$ locally with $\{w=0\}$. Then $\tilde u:=p^*u$ (and hence $d\tilde u$) defines a \textit{global single-valued} function (exact $1$-form) on $\widetilde{\R^n}$.  

Consider the inclusion map

\begin{equation}
    \label{eqn:iotadelta}
    \iota_{\delta }:\widetilde{\R^n}\to T^*\R^n,
    \quad
    \tilde x\mapsto (p(\tilde x),\delta y(\tilde x)),
\end{equation}
where $y(\tilde x)\in T^*_{p(\tilde x)}\R^n$ is the unique covector such that $p^*y(\tilde x)=d\tilde u(x)$. Since $p$ is a local diffeomorphism away from the branching locus $\Sigma$, and $du=0$ on $\Sigma$ (so $y\equiv 0$ on $\Sigma$), $\iota_\delta$ is well-defined. 

The image $L_{\delta}:=\operatorname{im}(\iota_\delta)$ is just the graph of the two-valued $1$-form $\delta du$ over the zero section $\R^n\subset T^*\R^n$. 

\begin{figure}[H]
    \centering
    \begin{tikzcd}
        \widetilde{\R^n} \arrow[d, "p"'] \arrow[r, "\iota_{\delta }"] & L_{\delta}\subset T^*\R^n \arrow[d, "\pi"] \\
        \R^n \arrow[r]                                             & \R^n                             
    \end{tikzcd}
    \caption{The multivalued graph $L_{\delta}=\operatorname{graph}(\delta du)$.}
    \label{fig:idaslg}
\end{figure}

Thus, the multivaluedness is only apparent from the viewpoint of the base. After passing to the double branch cover $\widetilde{\R}^n$, all the relevant objects become globally defined and single-valued. For notational convenience, we will often suppress the distinction between the multivalued object downstairs and its single-valued lift upstairs.

\begin{lem}\label{lem:reg+potential}Let $u$ be a non-degenerate harmonic function on $\R^n$ with $|du|^{-1}(0)=\Sigma$. Then for $\delta\ll 1$, the two-valued graph $L_\delta=\operatorname{graph}(\delta du)$ defines a connected closed exact embedded Lagrangian inside $\C^n$, smooth away from the branching locus $\Sigma$, and $C^{1,\alpha}$ around $\Sigma$. Moreover, with respect to the choice of Liouville form in \eqref{eqn:Cn-CYstr}, the Lagrangian potential of $L_{\delta}$ is given by
    \begin{equation}
        \label{eqn:Lagpotential}
        f_{L_{\delta}}=-\delta u+\frac{\delta}{2}\langle x,\nabla u\rangle.
    \end{equation}
\end{lem}

\begin{proof}
    We first examine the regularity of $L_\delta$ around any point $q\in \Sigma$. Let $(t,z)$ be a local Fermi coordinate centred at $q=(0,0)$. By the structure Theorem \ref{thm:structureofz2}, we have 
    $$du=d\Re({\bf{B}}(t)z^{3/2})+dE,$$
where $E\in C^{2,\alpha}_{\operatorname{loc}}\cap \mathcal{D}^{\infty,\alpha}_{\operatorname{loc}}$ with corresponding error estimates. Passing to the double branch cover $\widetilde{\R^n}$, we know locally $p:(t,w)\mapsto (t,z=w^2)$  around $q\in \Sigma$, with
\begin{equation}
    \begin{aligned}
    \label{eqn:pullbackdu}
        &  d\tilde u=p^*du=\Re\left((3{\bf B}(t)w^2+4w\partial_zE(t,w^2))dw\right)+\sum_{\ell=1}^{n-2}\left(\Re(\partial_{t_{\ell}}{\bf{B}}(t)w^3)+\partial_{t_\ell}E(t,w^2)\right)d{t_\ell}.
    \end{aligned}
\end{equation}
In the definition \eqref{eqn:iotadelta} of the inclusion map $\iota_{\delta}:\widetilde{\R^n}\to T^*\R^n$, $y(\tilde x)=y(t,w)\in T_{(t,z)}^*\R^n$ is characterised by the relation $p^* y(t,w)=d\tilde u|_{(t,w)}$. Write $y(t,w)=\Re(\textbf{Y}(t,w)dz)+\sum\limits_{1\leqslant \ell\leqslant n-2}y_\ell(t,w) dt_\ell$, for real-valued  $y_\ell,1\leqslant\ell\leqslant n-2$ and complex-valued $\textbf{Y}$, then

$$p^*y(t,w)=\Re(2w\textbf{Y}dw)+\sum_{1\leqslant\ell\leqslant n-2}y_\ell dt_\ell.$$
Comparing with \eqref{eqn:pullbackdu}, we find
\begin{equation}
    \label{eqn:graphfibrecoordinate}
    y(\tilde x)=y(t,w)=\Re\left(\left(\frac{3}{2}{\bf B}(t)w+2\partial_zE(t,w^2)\right)dz\right)+\sum_{\ell=1}^{n-2}\left(\Re(\partial_{t_{\ell}}{\bf{B}}(t)w^3)+\partial_{t_\ell}E(t,w^2)\right)d{t_\ell}\in T_{(t,z)}^*\R^n.
\end{equation}
As $\nabla E\in C^{1,\alpha}_{\operatorname{loc}}$ with
estimates $|\nabla E|\lesssim |z|^{1+\alpha}$, it immediately follows that $\tilde y(x)\in C^{1,\alpha}$. The non-degeneracy of $u$ (\ie ${\bf{B}}\in \Gamma(N_\Sigma^{-3/2})$ is nowhere vanishing) guarantees that the dominant term in \eqref{eqn:graphfibrecoordinate} is $\Re(\frac{3}{2}\textbf{B}wdz)$, the implicit function theorem implies that $\iota_\delta$ is a local immersion around any $q\in \Sigma$. The embeddedness follows from $|du|^{-1}(0)=\Sigma$. 

Now we prove the Lagrangian potential formula \eqref{eqn:Lagpotential} for $L_\delta$. Writing $u_k=\frac{\partial u}{\partial x_k}$, we compute
    \begin{equation*}
        \begin{aligned}
            \lambda|_{L_\delta}
            &=
            \sum_{k=1}^n
            \left.
            \left(
            \frac{1}{2}x_kdy_k-\frac{1}{2}y_kdx_k
            \right)
            \right|_{L_\delta} \\
            &=
            \delta\sum_{k=1}^n
            \left(
            \frac{1}{2}x_k du_k-\frac{1}{2}u_kdx_k
            \right)\\
            &=
            \delta d\left(-u+\frac{1}{2}\langle x,\nabla u\rangle\right).
        \end{aligned}
    \end{equation*}
    The above pointwise equality holds away from $\Sigma$. Since \eqref{eqn:Lagpotential} is continuous and vanishes on $\Sigma$ by Theorem \ref{thm:structureofz2}, the lemma follows from the $C^{1,\alpha}$-regularity of $L_\delta$. 
\end{proof}
\section{Uniqueness when $n=2$}
\label{sec:uniquen=2}
\begin{prop}
     Let $u$ be a $\Z_2$-harmonic function on $\C$ with non-empty branching locus
    $\Sigma=\{p_1,\ldots,p_m\}\subset \C$. Assume that ${\bf A}\equiv0$ at each point of $\Sigma$, and that $u$ has polynomial growth at infinity. Set
    $F(z)=\prod\limits_{k=1}^m(z-p_k)$. 
    Then the following holds. 
    \begin{enumerate}
        \item [\textnormal{(1)}] There exists a polynomial $G(z)\in \C[z]$ such that
    \begin{equation}
        \label{eqn:1-form-expression-forn=2}
        du=\Re\left(G(z)\sqrt{F(z)}\,dz\right).
    \end{equation}
        \item [\textnormal{(2)}] The $\Z_2$-harmonic $1$-form \eqref{eqn:1-form-expression-forn=2} is exact if and only if all real period integrals vanish, \ie
        \begin{equation}
            \label{eqn:exactconditionforn=2}
            \Re\int_{\gamma}G(z)\sqrt{F(z)}dz=0,
        \end{equation}
        for any closed curve $\gamma$ on the branched double cover $\widetilde X=\{(z,y):y^2=F(z)\}$.  
        \item [\textnormal{(3)}] In the case of quadratic growth, we have $m=2$, and $G(z)\equiv C\in \C^*$. The integrability condition \eqref{eqn:exactconditionforn=2} reduces to
        $$C(p_1-p_2)^2\in \R.$$
        Up to translation and rotation on $\C$, we can assume $p_1=-p_2=a>0$. Then
        \begin{equation}
            \label{eqn:answerforn=2}
            u=\frac{C}{2}\Re\left(z\sqrt{z^2-a^2}-a^2\log(z+\sqrt{z^2-a^2})\right),\quad C\in \R^*.
        \end{equation}
    \end{enumerate}
\end{prop}

\begin{proof}
    Away from $\Sigma$, we choose a local single branch of $u$. Then the complex $1$-form $\omega:=2\partial u $
    is holomorphic and satisfies $du=\operatorname{Re}\omega$. By definition, both $u$ and $\omega$ have the same $\Z_2$-monodromy as $\sqrt{F(z)}$ around any point in $\Sigma$. Therefore, $\frac{\omega}{\sqrt{F(z)}}$
    is a single-valued holomorphic $1$-form on $\C\setminus\Sigma$.
    
    We claim that it extends holomorphically across each $p_k$, $1\leqslant k\leqslant m$. Indeed, near $p_k$, write $\zeta=z-p_k$. Since ${\bf A}(p_k)=0$, the structure theorem \ref{thm:structureofz2} implies that
    $
        \omega=2\partial u=2\partial\operatorname{Re}({\bf{B}}(p_k)\zeta^{3/2})+O(|\zeta|^{1+\alpha}).
    $
     On the other hand, $\sqrt{F(z)}=\sqrt{\zeta}\sqrt{\prod\limits_{j\neq k}(z-p_j)}$. 
     The claim follows from Riemann's removable singularity theorem. Hence there is an entire function $G(z)$ such that
    $\omega=G(z)\sqrt{F(z)}dz.$
    
    On any contractible set not containing $\Sigma$, the standard Schauder estimates imply that $\nabla^ku$ also has polynomial growth at infinity, and hence $G(z)$ has polynomial growth. It follows from Liouville's theorem that $G\in \C[z]$. This completes the proof of item (1). 
    
    Moreover, if $G\in\C[z]$, then $\operatorname{Re}\omega=\operatorname{Re}\left(G(z)\sqrt{F(z)}\,dz\right)$ is a  $\Z_2$-harmonic $1$-form on $\C\setminus\Sigma$, with monodromy $-1$ around each point in $\Sigma$. Passing to the branched double cover $\widetilde{X}$, standard theory implies that it integrates to a globally defined $\Z_2$-harmonic function exactly when the real parts of all periods vanish on the branched double cover. Thus item (2) follows. 

    Finally, in the case of quadratic growth, tracking the growth of \eqref{eqn:1-form-expression-forn=2}, we find $\deg G+\frac{m}{2}=1$. Hence $m=2$ and $G\equiv C\in \C^*$. In this case, $\widetilde{X}\cong \C^*$, and the period integral is determined by the generator $\gamma\in H_1(\C^*)$ whose projection to the base $\C$ links $p_1$ and $p_2$,
    $$\Re\int_{\gamma}G(z)\sqrt{F(z)}dz=\Re \left(C\int_\gamma ydz\right)=\Re\left(\pm \frac{iC\pi}{4}(p_1-p_2)^2 \right).$$
    Item (3), together with \eqref{eqn:answerforn=2}, then follows from direct calculations.
\end{proof}
\section{Analytical Preparations}
\label{Sec:analyticalpreparations}
Let $n\geqslant 3$.
We start with several preliminary reductions:
\begin{itemize}
    \item Since the set of indicial roots of $\Delta_{\R^n}$ is $\{m,2-n-m:m\in \Z_{\geqslant0}\}$, by elliptic theory on asymptotically conical spaces \cite{Lockhart-McOwen}, we can find $\ell \in \R^n$ and $M_0\in \R$ such that 
    $$u\sim \pm\left({x^TMx}+x^T\ell-M_0+\mathcal{O}(|x|^{\mu})\right),\quad \forall 2-n<\mu<0,$$
    where $V(x)=x^TMx$ and $M$ is a non-degenerate symmetric matrix.  
    \item We can then use a translation of $\R^n$ and the non-degeneracy of $V$ to eliminate the linear term, and then rotate the coordinates so that the quadratic form $V(x)$ is diagonal. Thus, we can assume $M=\operatorname{diag}(M_1,\cdots,M_n)$ with $M_i\neq 0$ and $\ell=0$. 
    \item It follows from a standard blow-down argument that $M_1+\cdots M_n=0$, \ie $V(x)=\sum\limits_{k=1}^n M_kx_k^2$ is a harmonic polynomial. Applying standard weighted Schauder estimates, we obtain
   \begin{equation}
       \label{eqn:higher-order derivative estimates}
       \left|\nabla^k\left(u-\sum_{k=1}^n M_kx_k^2+M_0\right)\right|\leqslant C |x|^{\mu-k},\quad |x|\geqslant R_0\gg1,~k\in \Z_{\geqslant 0},~\forall 2-n<\mu<0.
   \end{equation}
 
\end{itemize}

The rest of this section is devoted to proving the non-degeneracy condition (\ie ${\bf B}\in \Gamma(N_\Sigma^{-3/2})$ is nowhere vanishing)  as well as $|du|^{-1}(0)=\Sigma$, by adapting Almgren's frequency function, with the caveat that we do not \textit{a priori} assume $u$ is a Dirichlet energy minimiser. 

\begin{prop}\label{prop:freqofdu}
    Let $\beta=du\in \mathcal{D}^{\infty,\alpha}_{\operatorname{loc}}(\mathcal{L}\otimes T^*\R^n)$ be a non-zero $\Z_2$-harmonic $1$-form. Then for each point $p\in \R^n$, define Almgren's frequency function of $\beta$ by
    $$I(p,r)=\frac{r D(p,r)}{H(p,r)},\quad \textit{where}~D(p,r):=\int_{B_r(p)}|\nabla \beta|^2,\quad H(p,r):=\int_{\partial B_r(p)}|\beta|^2,\quad \forall r>0.$$
    Then the following hold,
    \begin{enumerate}
        \item [\textnormal{(1)}] We have $H\in C^1(\R_{>0})$, and $D$ is absolutely continuous, with
        \begin{equation}
        \label{eqn:diffofDH}
            H'(p,r)=\frac{n-1}{r}H(p,r)+2D(p,r),\quad \forall r>0;\quad D'(p,r)=\int_{\partial B_r(p)}|\nabla \beta|^2,\quad \textit{for a.e. } r>0.
        \end{equation}
        In particular, $h(p,r)=r^{1-n}H(p,r)$ is non-decreasing.    

        \item [\textnormal{(2)}] The frequency function $I(p,r)$ is well-defined, non-decreasing, and locally absolutely continuous on $\R_{>0}$, with  
        \begin{equation}
        \label{eqn:dfreq}
            I'(p,r)=\frac{2r}{H(p,r)^2} 
        \left(\int_{\partial B_r(p)}|\nabla_\nu  \beta|^2\int_{\partial B_r(p)}|\beta|^2-\left(\int_{\partial B_r(p)}\langle\beta,\nabla_\nu \beta\rangle\right)^2\right)\geqslant0,
        \end{equation}
        where $\nu(x)=\frac{x-p}{|x-p|}, ~x\in \partial B_r(p)$.
        Consequently,
        \begin{equation}
            \label{eqn:HtoI}
            \frac{d}{dr}\log h(p,r)=\frac{2I(p,r)}{r}.
        \end{equation}
        One then defines $\operatorname{ord}_\beta(p):=\lim\limits_{r\to 0} I(p,r)$, which measures the vanishing order of $\beta$ at $p$. 
        \item [\textnormal{(3)}] If $I(p,r)\equiv I_0>0$ for $0<r<r_0$, then $\beta$ is $I_0$-homogeneous with respect to the centre $p$.
    \end{enumerate}
\end{prop}

\begin{proof}By the structure Theorem \ref{thm:structureofz2}, we have
the following asymptotic expansion near any point $q\in \Sigma$,
    \begin{equation}
        \begin{aligned}
        \label{eqn:asymptoteofbeta}
            \beta=d\Re({\bf{B}}(t)z^{3/2})+dE,\quad  \forall t\in \Sigma,
        \end{aligned}
    \end{equation}
    with error estimates $|\nabla E|\lesssim |z|^{1+\alpha}$ and $|\nabla^2 E|\lesssim |z|^\alpha$.  

Following \cite[Chapter 3]{DeLellis-Spadaro}, we first show the following stationary equation,
    \begin{equation}
        \label{eqn:first-variation-of-beta}
        2\int_{\R^n}\langle S_\beta,\nabla X^\flat\rangle-\int_{\R^n}|\nabla \beta|^2\operatorname{div}X=0,
    \end{equation}
    for any smooth compactly supported vector field $X$ on $\R^n$, where $S_\beta(\cdot,\cdot)=\langle \nabla_\cdot\beta,\nabla_\cdot\beta\rangle$ is a symmetric $(0,2)$-tensor and $\nabla X^\flat(\cdot,\cdot)=g(\nabla_\cdot X,\cdot)$ is a $(0,2)$-tensor. Here and below, we freely use the inner product on different associated bundles induced by the flat structure of $\mathcal{L}$ and the Euclidean metric on $\R^n$.

    Let $\chi(s)=\left\{\begin{array}{cc}
       0  &  s\leqslant 1,\\
       1  &  s\geqslant 2,
    \end{array}\right.$ be a cut-off function. 
    Set $\rho=d(\cdot,\Sigma)$, and consider the following test section
    $\varphi=\chi_\varepsilon\nabla_X\beta$ where $\chi_\varepsilon=\chi(\rho/\varepsilon)$ and $\varepsilon\ll1 $ is sufficiently small (to be specified below). The harmonicity of $\beta$ implies the stationarity equation away from $\Sigma$,
\begin{equation*}
    \begin{aligned}
        0&=\int_{\R^n} \langle \nabla \beta,\nabla(\chi_\varepsilon \nabla_X\beta)\rangle=\int_{\R^n}\chi_\varepsilon \langle S_\beta,\nabla X^\flat\rangle+\frac{1}{2}\int_{\R^n}\chi_\varepsilon\langle \nabla |\nabla \beta|^2,X\rangle+\int_{\R^n} S_\beta(\nabla \chi_\varepsilon,X)
        \\&=\int_{\R^n} \chi_\varepsilon \left(\langle S_\beta,\nabla X\rangle-\frac{1}{2}|\nabla \beta|^2\operatorname{div}X \right)+\int_{\varepsilon \leqslant\rho\leqslant2\varepsilon}  S_\beta(\nabla \chi_\varepsilon,X)-\frac{1}{2}|\nabla \beta|^2 \langle X,\nabla \chi_\varepsilon\rangle.
    \end{aligned}
\end{equation*}  
 We now need to prove the defect of the last term tends to zero as $\varepsilon\to 0+$. We remark that the na\"ive estimate $|\nabla \beta|\lesssim \rho^{-1/2}$ is not sufficient to control this error (even if we replace $\chi_\varepsilon$ by a logarithmic cut-off)
since 
$$\int_{\varepsilon}^{2\varepsilon} \underbrace{\rho^{-1}}_{|\nabla \chi_\varepsilon|}\cdot \underbrace{\rho^{-1}}_{|\nabla \beta|^2}\rho d\rho\sim \mathcal{O}(1).$$
The key observation is that the leading $\rho^{-1/2}$-terms in $\nabla\beta$ arise only from taking derivatives of $\beta$ transverse to $\Sigma$, and they indeed cancel in the total integrand via the following differential-geometric calculation. 

Let $(t_1,\cdots,t_{n-2},z)$ be a local Fermi coordinate system defined on an open neighbourhood $U\subset \R^n $ of any given $q\in \Sigma$. Then we have $\rho=d(\cdot,\Sigma)=|z|$ and $\nabla\rho=\partial_\rho$. The metric tensor behaves as $g=g_{\Sigma}+g_{\C}+\mathcal{O}(\rho)$ and hence $\nabla^g-\nabla^{g_{\Sigma}+g_\C}=\mathcal{O}(1)$. The upshot is the deviation of the metric structure in Fermi coordinates does not affect the leading term we aim to cancel. 

Around $q$, the asymptote \eqref{eqn:asymptoteofbeta} with derivative bounds on the error $E$ implies that
\begin{equation}\label{eqn:hessianu}
    ~\left\{\begin{array}{l}
        \nabla_{\partial_\rho} \beta=\frac{3}{4}\rho^{-1/2}\operatorname{Re}({\bf{B}}e^{3i\theta/2}) d\rho-\frac{3}{4}\rho^{-1/2}\operatorname{Im}({\bf{B}}e^{3i\theta/2})(\rho d\theta)+\mathcal{O}(\rho^{\alpha}),\\
        \nabla_{\rho^{-1}\partial_\theta} \beta=-\frac{3}{4}\rho^{-1/2}\operatorname{Im}({\bf{B}}e^{3i\theta/2}) d\rho-\frac{3}{4}\rho^{-1/2}\operatorname{Re}({\bf{B}}e^{3i\theta/2})(\rho d\theta)+\mathcal{O}(\rho^{\alpha}),\\
        \nabla_{\partial_{t_i}}\beta=\mathcal{O}(\rho^{\alpha}).
    \end{array}\right.
\end{equation}
Thus, 
\begin{equation*}
    \left\{\begin{array}{l}
        S_\beta(\partial_\rho,\partial_\rho)=\frac{9}{16}\rho^{-1}|{\bf{B}}|^2+\mathcal{O}(\rho^{\alpha-1/2})\\ 
        S_\beta(\partial_\rho,\rho^{-1}\partial_\theta)=\mathcal{O}(\rho^{\alpha-1/2}),\\
        |\nabla \beta|^2=\frac{9}{8}\rho^{-1}|{\bf B}|^2+\mathcal{O}(\rho^{\alpha-1/2}).
         
    \end{array}\right.
\end{equation*}
 Now expand $X$ in the chosen Fermi co-ordinates around $q$, say $X=X_\rho\partial_\rho +X_\theta\rho^{-1}\partial_\theta+X_{t_i}\partial_{t_i}$ with $|X_{\rho}|,|X_\theta|,|X_{t_i}|\sim \mathcal{O}(1)$. Then within the support of $\nabla \chi_\varepsilon\subset\{\varepsilon\leqslant \rho\leqslant 2\varepsilon\}$, we have
 \begin{equation*}
     \begin{aligned}
         &\quad S_{\beta}(\nabla\chi_{\varepsilon},X)-\frac{1}{2}|\nabla \beta|^2\langle X,\nabla\chi_\varepsilon\rangle=\frac{\chi'}{\varepsilon}\left(S_\beta(\partial_\rho,X)-\frac{1}{2}|\nabla \beta|^2\langle X,\partial_\rho\rangle\right)
         \\&=\frac{\chi'}{\varepsilon}(\underbrace{X_\rho S_{\beta}(\partial_\rho,\partial_\rho)-\frac{X_\rho}{2}|\nabla \beta|^2\langle \partial_\rho,\partial_\rho\rangle}_{\mathcal{O}(\rho^{\alpha-1/2})}+\underbrace{S_\beta(\partial_\rho,X-X_\rho\partial_\rho)}_{\mathcal{O}(\rho^{\alpha-1/2})}-\underbrace{\frac{1}{2}|\nabla \beta|^2\langle X-X_\rho\partial_\rho,\partial_\rho\rangle}_{\mathcal{O}(\rho^{-1}\cdot \rho)}
         \\&=\mathcal{O}(\rho^{\alpha-3/2}).
     \end{aligned}
 \end{equation*}
The above error estimate is then \textit{global} once we take $\varepsilon\leqslant \frac{1}{4}\operatorname{inj}_\Sigma$, where $\operatorname{inj}_\Sigma$ denotes the injectivity radius of the normal exponential map of the \textit{closed} codimension--$2$ submanifold $\Sigma \subset \R^n$. Then \eqref{eqn:first-variation-of-beta} follows from the basic fact that $\int_{\varepsilon}^{2\varepsilon}$ $\rho^{\alpha-3/2}\rho d\rho\lesssim \varepsilon^{1/2+\alpha}\to 0$. 

Now the conclusion follows by standard arguments as in \cite[Chapter 3]{DeLellis-Spadaro}, which we sketch for completeness. Using test vector fields $X=\phi_N(|x-p|)(x-p)$, with $\phi_N=\left\{ \begin{array}{lc}
        1 &  t\leqslant r-\frac{1}{N}\\ 
        N(r-t) & r-\frac{1}{N}\leqslant t\leqslant r
        \\ 0&~r\leqslant t
    \end{array} \right.$ up to smooth approximation and letting $N\to \infty$, the following two identities hold for a.e. $r>0$, 
    \begin{equation}
    \label{eqn:Rellich}
        \begin{aligned}
            (n-2)D(p,r)&=r\int_{\partial B_r(p)}|\nabla \beta|^2-2r\int_{\partial B_r(p)} |\nabla_\nu\beta|^2,
            \\D(p,r)&=\int_{\partial B_r(p)}\langle \beta,\nabla_\nu \beta\rangle. 
        \end{aligned}
    \end{equation}
    As $|\beta| \in C^{0,\alpha}_{\operatorname{loc}}\cap W^{1,2}_{\operatorname{loc}}$, $D(p,r),H(p,r)$ are well-defined absolutely continuous functions. The differentiation formulae \eqref{eqn:diffofDH} and \eqref{eqn:dfreq} (if $I$ is well-defined) follow from \eqref{eqn:Rellich}, and hence $h'(p,r)=2r^{1-n}D(r)\geqslant 0$. Therefore, if $H(p,r_0)=0$ for some $r_0>0$, then $\beta |_{B_{r_0}(p)}\equiv 0$. The unique continuation property then implies $\beta \equiv 0$ on $\R^n\setminus\Sigma$, and thus on the whole $\R^n$ by continuity, contradicting our assumption that $\beta$ is non-zero.  

    For the last item, using \eqref{eqn:dfreq} and \eqref{eqn:diffofDH}, we know $I\equiv I_0$ is equivalent to $\nabla_\nu \beta(p+x)=\frac{I_0}{|x|}\beta(p+x)$ for a.e. $x$. Integrating over the segment between $p+r_1x$ and $p+r_2x$ with $0<r_1<r_2<r_0$ and $|x|\leqslant 1$, we obtain the pointwise identity (by the continuity of $\beta$) $\beta(p+rx)=(\frac{r}{s})^{I_0}\beta (p+sx)$.  
\end{proof}

\begin{cor}\label{cor:freqofdu}
    For $u$ as in \eqref{eqn:higher-order derivative estimates}, the following holds.  
    \begin{enumerate}
        \item [\textnormal{(1)}] The leading asymptote ${\bf{B}}\in \Gamma(N_{\Sigma}^{-3/2})$ is nowhere vanishing. 
        \item [\textnormal{(2)}] The $\Z_2$-harmonic $1$-form $du$ does not vanish away from the branch set, \ie $|du|^{-1}(0)=\Sigma$. 
    \end{enumerate}
\end{cor}
\begin{proof}Since $u$ is asymptotic to a non-degenerate harmonic polynomial with degree $2$, a direct calculation implies
$$\operatorname{ord}_\infty(\beta):=\lim_{r\to \infty}I(p,r)=1,\quad \forall p\in \R^n.$$

\begin{itemize}
    \item To show item (1), employing the monotonicity of $I$ and $\operatorname{ord}_\infty(\beta)=1$, it suffices to prove the claim that $\operatorname{ord}_0(p_0)\geqslant 1+\alpha$ if ${\bf{B}}(p_0)=0$ for some $p_0\in \Sigma$. Suppose to the contrary that there exist $\varepsilon>0$ and $r_0>0$, such that 
    $I(p_0,r)\leqslant 1+\alpha-\varepsilon$ for $r\leqslant r_0$. Then \eqref{eqn:HtoI}
 implies that
 $$h(p_0,r)=h(p_0,r_0)\exp\left(-2\int_r^{r_0}\frac{I(p_0,s)}{s}ds\right)\geqslant Cr^{2(1+\alpha-\varepsilon)}.$$
 On the other hand, taking a Fermi coordinate system $(t,z)=(t,\rho e^{i\theta})$ centred at $p_0=(0,0)$ as in the proof of Proposition \ref{prop:freqofdu}, the asymptote of $u$ around $p_0$ implies that 
 $$h(p_0,r)=r^{1-n}\int_{\partial B_r(p_0)}|\beta|^2\leqslant C\sup_{\partial B_r(p_0)}|\beta|^2\leqslant C \max \{|t|^2\rho,r^{2(1+\alpha)}\}\leqslant Cr^{2(1+\alpha)},~~r\sim |t|+\rho\ll1$$
 where we use the fact that ${\bf B}\in \Gamma(N_\Sigma^{-3/2})$ is $C^\infty$ along $\Sigma$. Taking $r\to 0^+$, the above two estimates contradict each other.  
    \item For item (2), assume that $du$ vanishes at a point $p_0\in \R^n\setminus\Sigma$. Note that away from the branch set $\Sigma$, $u$ can locally be regarded as a single-valued harmonic function in the usual sense, so does $du$. It follows that $\lim\limits_{r\to 0}I(p_0,r)$ coincides with the vanishing order $m\geqslant 1$ of $du$ at $p_0$. Then the monotonicity of $I$, together with the linear growth at the infinity, forces $m=1$, $I(p_0,r)\equiv 1$ for $r>0$ and hence $du$ is $1$-homogeneous with respect to the centre $p_0$ by Proposition \ref{prop:freqofdu}. The asymptote at infinity then forces $p_0=0$ and $du=\pm dV$, and hence $\Sigma\subset |du|^{-1}(0)=\{0\}$, a contradiction. 
\end{itemize}
\end{proof}
\begin{rmk}[Compare with the uniqueness result in \cite{Sun-Pell,Haydys-Mazzeo-Takahashi-Examples}]
    Actually, one can show (similar to Proposition \ref{prop:freqofdu}, but much easier) that the frequency function is well-defined and non-decreasing with respect to $r>0$, for those $\Z_2$-harmonic functions we consider in the main Theorem \ref{mainthm} \textit{but with general polynomial growth} at infinity. 
    Then a very similar argument rules out $\Z_2$-harmonic functions with nontrivial branching locus, whose asymptote at infinity is modelled on an affine linear function.

\end{rmk}
\section{Floer-theoretic Part}
\label{Sec:floer-theory}
\subsection{Approximate special Lagrangian graph} 
From this subsection onward, we shall repeatedly use the geometric meaning of non-degenerate $\Z_2$-harmonic $1$-forms -- describing infinitesimal branched deformations of the special Lagrangian $\Pi_0=\R^n\subset \C^n$ as a two-valued graph. The following proposition illustrates this quantitatively.

\begin{prop}\label{prop:appslag}
    There exists a small $\delta_0=\delta_0(u)>0$, such that for all $0<\delta<\delta_0$, the two-valued graph $L_{\delta du}:=\operatorname{graph}(\delta du)$ is a connected,  exact, graded, $C^{1,\alpha}$--embedded Lagrangian in $\C^n$, asymptotically conical with rate $\mu<0$ to the union of two planes 
    $$\Pi^\delta_+=(e^{i\psi_1},\cdots,e^{i\psi_n})\Pi_0,\quad \Pi^\delta_-=(e^{-i\psi_1},\cdots,e^{-i\psi_n})\Pi_0,$$
    where $\psi_k=\arctan (2\delta M_k)\in (-\frac{\pi}{2},\frac{\pi}{2})$. With suitable normalisation, the Lagrangian angle $\theta_{L_{\delta du}}$ satisfies 
    \begin{equation}
        \label{eqn:estgrad}
        |\theta_{L_{\delta du}}|\leqslant C\delta^{1+\alpha},
    \end{equation}
    for some constant $C>0$ \textit{independent} of $\delta$. 
\end{prop}

\begin{proof}
    All statements except for the Lagrangian angle estimates follow from Lemma \ref{lem:reg+potential}. We concentrate on \eqref{eqn:estgrad}. Outside $\Sigma$, as $\nabla^2u$ is a two-valued smooth function, we define a grading of $L_{\delta du}$ away from $\Sigma$ to be
    \begin{equation}
    \label{eqn:defgrad}
        \theta_{L_{\delta du}}=\sum_{k=1}^n\arctan(\delta \lambda_k),
    \end{equation}
where $\lambda_1,\cdots,\lambda_n$ are the eigenvalues of $\nabla^2 u$ with $\Delta u=\sum\limits_{k=1}^n\lambda_k=0$. It is clear that 
    \begin{equation}
        \label{eqn:anglemod2pi}
        \theta_{L_{\delta du}}\equiv \arg \det (I+\sqrt{-1}\delta \nabla^2u),\quad (\operatorname{mod}~2\pi)
    \end{equation}
where we choose the argument function $\arg:\C\to [-\pi,\pi)$ with discontinuity locus $\R_{<0}$. We aim to show \eqref{eqn:defgrad} extends continuously to $\Sigma$ (indeed, taking the value zero on $\Sigma$), and satisfies the angle estimate \eqref{eqn:estgrad}. 

On the region $\rho:=d(\cdot,\Sigma)\gtrsim \delta$, the asymptote at infinity \eqref{eqn:higher-order derivative estimates} and local asymptote around $\Sigma$ imply that 
$$|\delta \lambda_k|\leqslant |\delta \nabla^2 u|\lesssim \delta\cdot \delta^{-1/2}=\delta ^{1/2} \ll 1,$$ provided $\delta_0$ is chosen sufficiently small. Moreover, using  $\left|\operatorname{arctan}x-x\right|\leqslant C|x|^3$ for $x\in [-\frac{1}{100},\frac{1}{100}]$, we obtain the phase estimate away from the branch locus,
$$|\theta_{L_{\delta du}}|=\left|\theta_{L_{\delta du}}-\sum_{k=1}^n \delta \lambda_k\right|\leqslant \sum_{k=1}^n|\delta \lambda_k|^3\lesssim \delta^3|\nabla^2 u|^3\lesssim \delta^{3/2}.$$

Now we concentrate on the region $\rho \lesssim \delta$. The heuristic is that two large eigenvalues of $\nabla^2u$ (corresponding to directions normal to $\Sigma$) are opposites to each other. Since $\arctan$ is odd, their sum is small. The remaining $(n-2)$ eigenvalues arise from the mixed term or the purely $\Sigma$--direction, and are inherently small of order $\rho^\alpha$ due to the local asymptote. 

Around any point $q\in \Sigma$, using the Fermi coordinate system $(t,z=\rho e^{i\theta})$ centred at $q=(0,0)$ as in the proof of Proposition \ref{prop:freqofdu}, we split $\nabla^2u$ as 
\begin{equation}
\label{eqn:Hessunearbranch}
    \nabla^2u=N+S,\quad N=\left(\begin{array}{ccc}
    0_{(n-2)\times (n-2)}& &\\ &\frac{3}{4}\rho^{-1/2}\operatorname{Re}({\bf{B}}e^{3i\theta/2}) & -\frac{3}{4}\rho^{-1/2}\operatorname{Im}({\bf{B}}e^{3i\theta/2}) \\
    &-\frac{3}{4}\rho^{-1/2}\operatorname{Im}({\bf{B}}e^{3i\theta/2}) & -\frac{3}{4}\rho^{-1/2}\operatorname{Re}({\bf{B}}e^{3i\theta/2})
\end{array}\right),\quad S=\mathcal{O}(\rho^\alpha),
\end{equation}
with respect to $\partial_{t_1},\cdots,\partial_{t_{n-2}},\partial_\rho,\rho^{-1}\partial_\theta$ (\cf \eqref{eqn:hessianu}). It follows that $N$ admits eigenvalue $0$ with multiplicity $n-2$ corresponding to the $t$-directions, and $\pm \lambda$ with $\lambda=\frac{3}{4}\rho^{-1/2}|\bf{B}|$ in the normal direction, and $\det(I+\sqrt{-1}\delta N)=1+\delta^2\lambda^2>0$, with $\|(I+\sqrt{-1}\delta N)^{-1}\|\leqslant 1$. The geometric observation is that as one approaches $\Sigma$, the error $S=\mathcal{O}(\rho^\alpha)$ becomes negligible, the Lagrangian angle has a well-defined limit which vanishes on $\Sigma$ -- the tangent space of $L_{\delta du}$ is purely vertical in the transverse direction to $\Sigma$ and hence a special Lagrangian plane. 

We then show the quantitative bound $|\theta_{L_{\delta du}}|\lesssim \delta \rho^\alpha$. Set $S'=\sqrt{-1}\delta(I+\sqrt{-1}\delta N)^{-1}S$, with the estimate $\|S'\|\lesssim \delta\rho^{\alpha}$. Then
$$\arg\det(I+\sqrt{-1}\delta\nabla^2 u)=\arg\det\left((I+\sqrt{-1}\delta N)(I+S')\right)=\arg\det (I+S'),$$
which gives the phase estimate $|\arg \det (I+S')|\lesssim \|S'\|\lesssim\delta \rho^\alpha\lesssim \delta^{1+\alpha}$, and hence stays away from the discontinuity line $\R_{<0}\subset \C$ of the argument function $\arg$. If we further choose $\delta_0\leqslant \frac{1}{4}\operatorname{inj}_\Sigma$, then the estimate above becomes global as before. 

Since both $\theta_{L_{\delta du}}$ and $\arg\det (I+\sqrt{-1}\delta\nabla^2 u)$ are continuous and small in the region $\rho \sim \delta$, they must coincide with each other in the region $\rho \lesssim \delta$ by \eqref{eqn:anglemod2pi} and a simple topological argument. The continuous extension of $\theta_{L_{\delta du}}$ to $\Sigma$ then follows from the above discussion; one could also check by passing to the double branch cover as in Lemma \ref{lem:reg+potential}, which we omit here.
\end{proof}
\begin{rmk}
    The estimate \eqref{eqn:estgrad} is not optimal but will be sufficient for our later use; the crucial point is that we obtain the extra $\delta^\alpha$ factor due to the $\Z_2$-harmonicity and non-degeneracy of $u$. One could try to refine the estimate by using the polyhomogeneous expansion of $u$ around $\Sigma$, distinguishing the directions of derivatives (tangential or normal to $\Sigma$), and using a more refined division of regions, say $\rho \sim \delta^2$ as the characteristic scale where $|\delta \nabla^2 u|\sim \mathcal{O}(1)$. 
\end{rmk}
 
\subsection{Floer degree formula and Holomorphic Strip}
To warm up, we recall several basics in Lagrangian Floer theory (\cf \cite[Section 2]{IJS}) in this section. Given two graded Lagrangians $L,L'$ inside an ambient symplectic Calabi-Yau manifold $M$ with Lagrangian angles $\theta_L,\theta_{L'}$, at a transverse intersection point $p$, we put 
\begin{equation}
\mu_{L,L'}(p)= \frac{1}{\pi}\left( \sum_{k=1}^n \alpha_k + \theta_L(p)- \theta_{L'}(p) \right) ,  
\end{equation}
where, within $T_pM$, we identify
\[
T_pL\cong \R^n \subset \C^n ,\quad T_p L'\cong (e^{i\alpha_1},\ldots e^{i\alpha_n})\R^n,\quad 0<\alpha_i<\pi.
\]

Specialised to graphs of harmonic $1$-forms with \textit{bounded derivatives}, we have

\begin{lem}\label{lem:floerdegreeforgraph} 
      Let $F_1,F_2$ be two (single-valued) $C^\infty$-functions defined near $0\in \R^n$, with $|\nabla^2 F_1|,|\nabla^2 F_2|\leqslant 1$. Suppose that $dF_1(0)=dF_2(0)=p$, and $\nabla^2(F_2-F_1)(0)$ is non-degenerate. Then
      for $0<\tau\leqslant \tau_0(n)\ll 1$, the Floer degree at $(0,\tau p)$ agrees with the negative index of $\nabla^2(F_2-F_1)(0)$, 
      $$\mu_{L_{\tau dF_1},L_{\tau dF_2}}(0,\tau p)=\operatorname{ind}(\nabla^2(F_2-F_1)(0)),$$ where the Lagrangian angles of $L_{\tau dF_i}$ are defined to be $\theta_i:=\sum\limits_{k=1}^n \arctan (\tau \lambda_k^i)$ for the eigenvalues $\lambda_1^i,\cdots,\lambda_n^i$ of $\nabla^2F_i$ near $0$, and $i=1,2$.    
   \end{lem}
   \begin{proof} The main effect of compressing two graphs by $\tau\leqslant \tau_0\ll1$ is that we can freely use the identity 
   $$\arg \det(I+\sqrt{-1}\tau A)=\sum_{k=1}^n\arctan (\tau \mu_k), $$
for the eigenvalues $\mu_1,\cdots,\mu_n$ of a real symmetric matrix $\|A\|\lesssim 1$, and $\arg :\C\setminus \R_{<0}\to (-\pi,\pi)$.  
   
       Writing $A_i=\nabla^2F_i(0)$, the tangent plane at $(0,\tau p)$ is given by
       $$T_i:=T_{(0,\tau p)}L_{\tau dF_i}=\operatorname{Graph}(\tau\nabla^2F_i(0))=(I+\sqrt{-1}\tau A_i)\R^n,\quad i=1,2, $$
       where the origin of $\R^n$ corresponds to $(0,\tau p)$ in the original coordinates. 
       
       Then $U_1:=(I+\tau^2 A_1^2)^{-1/2}(I-i\tau A_1)\in U(n)$ satisfies
       $$U_1T_1=\R^n,\quad |U_1-Id|\lesssim \tau|A_1|\leqslant\tau,\quad \arg \det U_1=-\theta_{1}(0).$$
       To write $U_1T_2$ as a graph over $\R^n$, we calculate that 
       $$U_1T_2=(I+\tau^2A_1^2)^{-1/2}(I+\tau^2A_1A_2+\sqrt{-1}\tau (A_2-A_1))\R^n=(X_{\tau}+\sqrt{-1}\tau Y_\tau)\R^n,$$
       where $X_\tau=(I+\tau^2A_1^2)^{-1/2}(I+\tau^2A_1A_2)$ is invertible provided that $\tau\leqslant \frac{1}{2}$ and $Y_\tau=(I+\tau^2A_1^2)^{-1/2}(A_2-A_1)$. Hence $U_1T_2=\operatorname{Graph}(\tau Y_{\tau}X_{\tau}^{-1})$ over $U_1T_1=\R^n$. Since $A_2-A_1$ is non-degenerate, so is $Y_{\tau}X_{\tau}^{-1}$. We compute 
        \begin{equation}
        \begin{aligned}
            \label{eqn:symmetricgraphmatrix}
         X_\tau^TY_\tau&=(I+\tau^2A_2A_1)(I+\tau^{2}A_1^{2})^{-1}(A_2-A_1)
               \\&=(I+\tau^2(A_1+A_2-A_1)A_1)(I+\tau^2A_1^2)^{-1}(A_2-A_1)
               \\&=(A_2-A_1)+\tau^2(A_2-A_1)A_1(I+\tau^2A_1^2)^{-1}(A_2-A_1),
            \end{aligned}
        \end{equation}
       and 
       \begin{equation}
           \begin{aligned}
               Y_\tau^T X_\tau&=(A_2-A_1)(I+\tau^2A_1^2)^{-1}(I+\tau^2A_1A_2)
               \\&=(A_2-A_1)(I+\tau^2A_1^2)^{-1}(I+\tau^2A_1(A_1+A_2-A_1))
               \\&=(A_2-A_1)+\tau^2(A_2-A_1)(I+\tau^2A_1^2)^{-1}A_1(A_2-A_1)
               \\&=(A_2-A_1)+\tau^2(A_2-A_1)A_1(I+\tau^2A_1^2)^{-1}(A_2-A_1)\quad (A_1,(I+\tau^2A_1^2)^{-1} \textnormal{commute})
               \\&=X_\tau^TY_\tau.
           \end{aligned}
       \end{equation}
       It follows that $Y_\tau X_\tau^{-1}$ is symmetric. 
       Now, writing $\tau \mu_1,\cdots,\tau\mu_n$ as the eigenvalues of $\tau Y_\tau X_\tau^{-1}$ with respect to an orthogonal basis of $U_1T_1=\R^n$, we have 
       $$U_1T_2=(e^{i\alpha_1},\cdots,e^{i\alpha_n})\R^n,\quad \alpha_k=\left\{\begin{array}{ll}
          \arctan(\tau\mu_k),  & \mu_k>0, \\
            \arctan(\tau\mu_k)+\pi, & \mu_k<0,
       \end{array}\right.\quad 1\leqslant k\leqslant n.$$
       Thus, 
       $$\mu_{L_{\tau dF_1},L_{\tau dF_2}}(0,\tau p)=\frac{1}{\pi}\left(\sum_{k=1}^n \alpha_k+\theta_1(0)-\theta_2(0)\right)=\operatorname{ind}(Y_\tau X_{\tau}^{-1}),$$
       as $\theta_2(0)-\theta_1(0)=\arg\det (U_1T_2)=\arg\left(\det(I+\sqrt{-1}\tau Y_\tau X_\tau^{-1})\det(X_\tau)\right)=\sum\limits_{k=1}^n\arctan(\tau\mu_k)$, whenever $\tau\leqslant \tau_0\ll1$. It remains to show 
       $$\operatorname{ind}(Y_{\tau}X_{\tau}^{-1})=\operatorname{ind}(A_2-A_1).$$
       Indeed, for $0<\tau\leqslant\tau_0(n)\ll1$, as the index is preserved by the congruence operation, we obtain \begin{equation*}
           \begin{aligned}
               \operatorname{ind}(Y_{\tau}X_{\tau}^{-1})&=\operatorname{ind}(X_\tau^T Y_\tau)\quad \quad\quad (\textnormal{congruence by} ~ X_\tau)
               \\&=\operatorname{ind}((A_2-A_1)+\tau^2(A_2-A_1)A_1(I+\tau^2A_1^2)^{-1}(A_2-A_1))\quad \textnormal{by (\ref{eqn:symmetricgraphmatrix})}
               \\&=\operatorname{ind}(\underbrace{(A_2-A_1)^{-1}}_{\textnormal{norm of all eigenvalues}\geqslant C(n)} +\underbrace{\tau^2A_1(I+\tau^2A_1^2)^{-1})}_{\textnormal{matrix norm}\leqslant C\tau^2}\quad (\textnormal{congruence by} ~ (A_2-A_1)^{-1})
                \\&= \operatorname{ind}((A_2-A_1)^{-1})=\operatorname{ind}(A_2-A_1).
           \end{aligned}
       \end{equation*}
The Lemma follows.        
   \end{proof}
   
Returning to global theory, the following theorem asserts the existence of a holomorphic strip of two \textit{compact} Lagrangians isomorphic in $D^bFuk$.
\begin{prop}[\textup{\cite[Theorem 2.15]{IJS}}]\label{prop:holomorphicstrip}
	Let $(M,\omega), \lambda$ be a symplectic Calabi-Yau Liouville manifold of dimension $2n$, and consider the derived Fukaya category $D^bFuk(M)$ of smooth, embedded, compact, exact, graded Lagrangians in $(M,\omega)$ with $\Z_2$ coefficients. Let $L, L'$ be transversely intersecting Lagrangians in $M$ which are isomorphic as objects of $D^bFuk(M)$. Let $J$ be an almost complex structure on $M$ compatible with $\omega$ and convex at infinity.

	Then there exist $p,q\in L\cap L'$ with $\mu_{L,L'}(p)=0$ and $\mu_{L',L}(q)=n$, and a (possibly broken) $J$-holomorphic strip $S$ with boundary in $L\cup L'$ and corners at $p,q$. Moreover, we can require the strip to pass through any prescribed point $r\in L$ (resp. $L'$). The area of the strip is 
	\[
	\int_{S}\omega= (f_L- f_{L'})(q)- (f_{L}- f_{L'})(p),
	\]
	where $f_L, f_{L'}$ denote the Lagrangian potentials.
\end{prop}

\subsection{Partial Compactification}
\label{sec:partialcompact}
 We now specialise the partial compactification procedure in Imagi--Joyce--Oliveira dos Santos uniqueness \cite{IJS} to our situation. We emphasise that all the constructions below depend on the small parameter $\delta\ll 1$. As in Proposition \ref{prop:appslag}, we consider two transverse Lagrangian planes $\Pi_\pm^\delta$ as asymptotic planes of the approximate special Lagrangian $L_\delta=\operatorname{graph}(\delta du)$. 
   
   We introduce the following new coordinate system:
    \begin{equation}
        \label{eqn:coorchangepartialcompact}
        x_k:=\operatorname{Re}(e^{i\psi_k}z_k)-\cot(\phi_k)\Im (e^{i\psi_k}z_k),\quad y_k:=\Im (e^{i\psi_k}z_k),~1\leqslant k\leqslant n,
    \end{equation}
    where $\phi_k=\left\{\begin{array}{ll}
      2\psi_k,   &  \psi_k>0, \\
       \pi+2\psi_k,  & \psi_k<0,
    \end{array}\right.\in (0,\pi)$.
    
    Then $\omega=\sum\limits_{k=1}^n dx_k\wedge dy_k$, and $\Pi_-^\delta=\{(x_1,\cdots,x_n,0):x_k\in \R\},~\Pi_+^\delta=\{(0,y_1,\cdots,y_n):y_k\in \R\}$. Regarding $(x_1,\cdots,x_n,y_1,\cdots,y_n)\in \C^n$ as $-\sum\limits_{k=1}^n y_kdx_k\in T^*_{x}\Pi_-^\delta$ and $\sum\limits_{k=1}^n x_kdy_k \in T^*_y\Pi_+^\delta$, we identify $(\C^n,\omega)\cong T^*\Pi_-^\delta\cong T^*\Pi_+^\delta$ as symplectic manifolds. 

    Topologically, we write $S_\pm^n=\Pi_\pm^\delta\cup \{\infty_\pm\}$ for the one-point compactifications of $\Pi_-^\delta,\Pi_+^\delta \cong \R^n$. We define coordinates $(\tilde x_1,\cdots ,\tilde x_n)\in S_{-}^n\setminus\{0\}$ and $(\tilde y_1,\cdots ,\tilde y_n)\in S_{+}^n\setminus\{0\}$ making $S_-^n$ and $S_+^n$ \textit{diffeomorphic} to $S^n$, 
    \begin{equation*}
        \tilde x_k(p)=\left\{\begin{array}{ll}
           0,  &  p=\infty_-\\
           \frac{x_k}{\sqrt{\sum\limits_{j=1}^n x_j^2}\log\left(1+\sum\limits_{j=1}^n x_j^2\right)},  & p=(x_1,\cdots,x_n)\in \Pi_-^\delta\setminus\{0\}
        \end{array}\right.
    \end{equation*}
    and 
    \begin{equation*}
        \tilde y_k(p)=\left\{\begin{array}{ll}
           0,  &  p=\infty_+\\
           \frac{y_k}{\sqrt{\sum\limits_{j=1}^n y_j^2}\log\left(1+\sum\limits_{j=1}^n y_j^2\right)},  & p=(y_1,\cdots,y_n)\in \Pi_+^\delta\setminus\{0\}.
        \end{array}\right.
    \end{equation*}
    Here the logarithmic term ensures that any asymptotic behaviour with negative polynomial rate along the ends of $\Pi_\pm^\delta$ extends smoothly to $\infty_\pm$ respectively.

    We then form the partial compactification \[
M= \C^n \sqcup T^*_{\infty_-} S_-^n\sqcup T^*_{\infty_+} S_+^n,
\]
as the plumbing of the cotangent bundles $T^*S_-^n$ and $T^* S_+^n$ along the common open set $\C^n$. The manifold $M$ inherits the natural symplectic structure $\omega$; however, the Liouville form $\lambda$ on $\C^n$ does not extend to $M$. Instead, we define a family (with respect to the parameter $T\gg 1$) of adapted Liouville forms $\tilde \lambda$ on $\C^n$ as follows.   

Take a cut-off function $\eta:\R\to [-1,1]$, with $|\eta'|=\mathcal{O}(T^{-1})$ such that $$\eta(t)=\left\{\begin{array}{ll}
     -1,& t\leqslant -2T, \\
     0, & -T\leqslant t\leqslant T,\\
     1,
&t\geqslant 2T,\end{array}\right.$$ and 
define a smooth 1-form
$\widetilde{\lambda}$ on $\C^n$ by
$$
\widetilde{\lambda}=\lambda+dh,\quad \lambda= \frac{1}{2} \left(\sum_{k=1}^n x_kdy_k-\sum_{k=1}^n y_kdx_k\right),\quad h= -\frac{1}{2} \eta\left( \sum_{k=1}^n x_k^2-\sum_{k=1}^n y_k^2   \right)\sum_{k=1}^n x_ky_k,
$$
then $d\widetilde{\lambda}=\omega$, and $\widetilde{\lambda}$ agrees with the natural Liouville forms on the cotangent bundles asymptotically, so it extends to a smooth Liouville form on $M$. 

Similarly, we can define  modified versions $\widetilde{J}, \widetilde{\Omega}$ for $J$ and $\Omega$, which extend smoothly to $M$,  so that $(M,\omega)$ is symplectically Calabi-Yau. 
We can arrange $\widetilde{J}=J$ and $\widetilde{\Omega}=\Omega$ when $-T\leqslant \sum\limits_{k=1}^n (x_k^2-y_k^2)\leqslant T$ for $T\gg 1$. 
We choose $\widetilde{J}$ so that $\widetilde{J}(T_{\infty_\pm}S_\pm^n)=T_{\infty_\pm}(T^*_{\infty_\pm}S_\pm^n)$ so that $S_\pm^n, T^*_{\infty_\pm}S_\pm^n$ have constant phase, with 
\begin{equation}
    \label{eqn:compactgrading}
    \theta_{S_\pm^n}=\pm\frac{1}{2}\sum_{k=1}^n\phi_k,\quad  \theta_{ T^*_{\infty_\pm}S_\pm   }= \pm\frac{1}{2}\sum_{k=1}^n\phi_k+\frac{n\pi}{2}.
\end{equation}
We shall consider a \textit{family} of such structures parametrised by $T\gg 1$, where the effect of taking $T\gg 1$ is similar to ``stretching the neck", such that the symplectic Calabi--Yau structure $(M,\omega)$ agrees with the standard structure of $\C^{n}$ on a larger and larger domain $-T\leqslant \sum\limits_{k=1}^n (x_k^2-y_k^2)\leqslant T$. 

Under the above framework, for a connected, closed, exact, graded Lagrangian $L$ asymptotic to $\Pi_\pm^\delta$ with rate $\mu<0$, potential $f_L$, and grading $\theta_L$ normalised by
$$\lim_{L\ni x\to \infty \textnormal{~along }\Pi_-^\delta}\theta_L(x)=-\sum_{k=1}^n\psi_k,$$
the asymptotic rate $\mu<0$ ensures that the following limits exist
$$\lim_{L\ni x\to \infty \textnormal{~along }\Pi_-^\delta} f_L(x)=c_-, ~~\lim_{L\ni x\to \infty \textnormal{~along }\Pi_+^\delta} f_L(x)=c_+,$$
and we then associate $L$ with the analytical invariant $A(L):=c_+-c_-$. 

Its compactification $\bar L=L\cup\{\infty_\pm\}\subset (M,\omega)$ is a smooth compact Lagrangian, intersecting the two cotangent fibres at infinity $T_{\infty_\pm}^*S_\pm$ transversely, with potential $f_{\bar L}$ and grading $\theta_{\bar L}$,
\begin{equation}\label{eqn:potentialcompact}
        f_{\bar L}(p)=\left\{\begin{array}{ll}
           f_L(p)+h(p),  & p\in L, \\
            c_\pm, & p=\infty_\pm,\\
        \end{array}\right.\quad \theta_{\bar L}(\infty_-)=-\sum_{k=1}^n\psi_k,~\theta_{\bar L}(\infty_+)=\lim_{L\ni x\to \infty \textnormal{~along }\Pi_+^\delta}\theta_L.
\end{equation}

\begin{rmk}\label{rmk:potentialinv}
Given \eqref{eqn:potentialcompact}, we observe that the Novikov exponent (potential difference) at (transverse) intersection points inside $\C^n$ of two Lagrangians $ L$ and $L'$
is invariant under the above partial compactification, \ie 
$$(f_{\bar L}-f_{\bar L'})|_{\bar L\cap \bar L'\cap \C^n}=(f_L-f_{L'})|_{L\cap L'}.$$
Moreover, in $\C^n$, the Floer degree also remains the same, as it is intrinsically characterised by the Lagrangian Grassmannian bundle, hence independent of the choice of complex volume form. 
\end{rmk}

Based on the work of Abouzaid and Smith \cite{Abouzaid-Smith}, Imagi--Joyce--Oliveira dos Santos proved the following classification result: 
\begin{thm}{\textup{\cite[Theorem 2.17, Corollary 2.18]{IJS}}} \label{thm:classifyaclag} 
Let 
$$\Pi_0=\R^n,\quad \Pi_\phi=(e^{i\phi_1},\cdots e^{i\phi_n})\R^n,\quad 0<\phi_k<\pi,\quad 1\leqslant k\leqslant n.$$
For a closed, exact, graded Lagrangian $L\subset \C^n$ asymptotic to $\Pi_0\cup \Pi_\phi$ with rate $\mu<0$, with grading normalisation $\lim\limits_{L\ni x\to \infty \textnormal{~along }\Pi_0}\theta_L(x)=0$, we have 
\begin{enumerate}
    \item [\textnormal{(1)}] $H^*(L,\Z_2)\cong H^*(S^{n-1}\times \R,\Z_2)$ and hence $L$ is connected. 
    \item [\textnormal{(2)}] along the end of $L$ asymptotic to $\Pi_\phi$, either $\theta_{L}\to\sum\limits_{k=1}^n\phi_k-\pi$ or $\theta_{L}\to \sum\limits_{k=1}^n\phi_k-(n-1)\pi$. 
    \item [\textnormal{(3)}] the isomorphism class of $\bar L$ in $D^bFuk(T^*S^n\#T^*S^n)$ is determined by the choice of the above asymptote.
\end{enumerate}
\end{thm}
Specialising to our situation, it follows that
\begin{cor}\label{cor:indexv}
    The asymptotic harmonic polynomial $\sum \limits_{k=1}^nM_kx_k^2$ of $u$ has index $1$ or $n-1$. 
\end{cor}

\begin{proof}
    For $\delta \ll 1$ as in Proposition \ref{prop:appslag}, we use  $U_\delta=\operatorname{diag}(e^{i\psi_1},\cdots,e^{i\psi_n})\in U(n)$ with $|U_\delta-Id|\lesssim \delta,~ \arg \det U_\delta=\mathcal{O}(\delta^3)$ to transfer $L_{\delta du}$ into the Imagi--Joyce--Oliveira dos Santos' setting Theorem \ref{thm:classifyaclag}, thus obtaining a closed exact graded Lagrangian $L$ asymptotic to  $\Pi_0\cup\Pi_\phi$, where 
    $$\phi=(\phi_1,\cdots,\phi_n)\in (0,\pi)^n,\quad \phi_k=\left\{\begin{array}{ll}
         2\arctan (2\delta M_k),& M_k>0, \\
         \pi+2\arctan (2\delta M_k),&  M_k<0,
    \end{array}\right.\quad \theta_{L}=\theta_{L_{\delta du}}+\arg\det U_\delta =\mathcal{O}(\delta^{1+\alpha}).$$ 

    Similar to Proposition \ref{prop:perturb} below, we can use Hamiltonian perturbation supported near the branching locus to smooth $L$ (still denoted as $L$), without changing the asymptotic planes at infinity, thus fitting into the derived Fukaya category of \textit{smooth} Lagrangians. Applying item (2) of Theorem \ref{thm:classifyaclag} to $L$, we obtain that (without loss of generality, we can assume the first alternative holds), 
      
     $$\sum_{k=1}^{n}\phi_k-\pi=\sum\limits_{k=1}^n2\arctan (2\delta M_k)+(m-1)\pi=\lim_{L \ni x\to \infty \operatorname{along} \Pi_\phi}\theta_{L}=\mathcal{O}(\delta^{1+\alpha}),$$
     where $m:=\#\{1\leqslant k\leqslant n:M_k<0\}=\operatorname{ind}V$.
     Since $\sum\limits_{k=1}^nM_k=0$ and $\operatorname{arctan}(x)=x+\mathcal{O}(x^3)$ for $|x|\ll 1$, we conclude $m=1$ by taking $\delta\to 0^+$. Similarly, the second alternative implies $\operatorname{ind}V=n-1$. 
\end{proof}
From now on, we can assume that $V$ has index $1$ and $M_1,\cdots,M_{n-1}>0$ with $M_n=-\sum\limits_{k=1}^{n-1}M_k<0$ (\cf item (4) of Remark \ref{rmk:scaling}).

\subsection{Hamiltonian Perturbation} 

In this section, we establish a Hamiltonian perturbation result for two multivalued graphs defined by two \textit{distinct} non-degenerate $\Z_2$-harmonic $1$-forms. Let $v$ be one of the $\Z_2$-harmonic functions arising from the Lawlor necks, with branch set $\Sigma'$ and \textit{the same asymptotic quadratic form} as $u$ at infinity. The goal is to remove the degree $0$ or $n$ intersections in the region, away from a small neighbourhood of the branching locus and the infinity.

We write $(L_\delta,L_\delta')=( L_{\delta du},L_{\delta dv})$ with partial compactification $(\bar L_\delta,\bar L_\delta')$ inside $(M,\omega)$ (\cf Section \ref{sec:partialcompact}) and let $\rho=d(\cdot,\Sigma)$ and $\rho'=d(\cdot,\Sigma')$ be the distance functions to $\Sigma,\Sigma'$ on $\R^n$ respectively. 

\begin{prop}\label{prop:perturb}
    Suppose that $du,dv$ are distinct non-degenerate $\Z_2$-harmonic $1$-forms. Then there exists $\Lambda_0=\Lambda_0(u,v)\gg 1$ such that for any $\Lambda>\Lambda_0$ and $\delta< \delta_0(u,v,\Lambda)$, one can find a family of smooth Lagrangians $\bar L_{\delta }^{\dagger}\subset M$ as small Hamiltonian perturbations of $\bar L_{\delta}$, parametrised by $\varepsilon\in (0,\varepsilon_0)$ for some $\varepsilon_0=\varepsilon_0(\delta,u,v,\Lambda)\ll 1$, such that for generic $\varepsilon\in (0,\varepsilon_0)$,
    \begin{enumerate}
        \item [\textnormal{(1)}](Global Transversality) The two Lagrangians $\bar L_{\delta}^\dagger$ and $\bar L_{\delta}'$ intersect transversely in $M$; 
        \item [\textnormal{(2)}] (No degree $0$ or $n$ intersections away from the common near-branch region and infinity) Any intersection of $\bar L_{\delta}^\dagger $ and $\bar L_{\delta}'$ within $\pi^{-1}(U_{\Lambda})$ has Floer degree between $1$ and $n-1$, where
        $$U_\Lambda =B_\Lambda(0)\setminus S_\Lambda,\quad S_{\Lambda}:=\{q\in \R^n:\rho(q)\leqslant \Lambda^{-1}\textnormal{~and~}\rho'(q)\leqslant \Lambda^{-1/2},\textnormal{~or~}\rho'(q)\leqslant \Lambda^{-1}\textnormal{~and~}\rho(q)\leqslant \Lambda^{-1/2}\}\subset \R^n.$$
        
        \item [\textnormal{(3)}] The family of Hamiltonian perturbations $\bar L_{\delta}^\dagger$ converges to $\bar L_\delta$ as $\varepsilon\to 0^+$ in the $C^1$-sense, with convergence of the perturbed potentials and gradings in the $C^0$-sense, \ie
        $$|f_{\bar L_{\delta}^\dagger}-f_{\bar L_{\delta}}|\leqslant \Psi(\varepsilon|\delta,u,v,\Lambda),\quad |\theta_{\bar L_{\delta}^\dagger}-\theta_{\bar L_{\delta}}|\leqslant \Psi(\varepsilon|\delta,u,v,\Lambda)$$
        where $\Psi(\varepsilon|C_1,C_2,\cdots)$ denotes a constant depending on $\varepsilon,C_1,C_2,\cdots$ that tends to $0$ when $C_1,\cdots $ are fixed and $\varepsilon \to 0$. 
        
    \end{enumerate}
    Moreover, when $v$ is replaced by $\pm V$, the asymptotic quadratic form of $u$ at infinity, the above conclusion still holds, and item \textnormal{(2)} can be strengthened to the assertion that there are no degree $0$ or $n$ intersection points within $\pi^{-1}(B_\Lambda(0)).$
\end{prop}
\begin{proof}
   \textbf{Division of regions and strategy.} We concentrate on item (2), as items (1) and (3) can be achieved by generic arbitrarily small Hamiltonian perturbations.  Let $U_\Lambda'$ be the complement of the $\Lambda^{-1}$-neighbourhood of $\Sigma\cup \Sigma'$, \ie
   $$U_{\Lambda}'=B_\Lambda(0)\cap \{q:\rho(q)> \Lambda^{-1}~\textnormal{and}~\rho'(q)> \Lambda^{-1}\}\subset \R^n.$$ For $p^\dagger \in  L_\delta^{\dagger}\cap  L'_\delta\cap \pi^{-1}(U_\Lambda)$, we divide the proof of $\mu_{ L_\delta^{\dagger},  L'_\delta}(p) \in \{ 1,\cdots,n-1\}$ into four cases with different treatments as follows: 
   
   \begin{itemize}
       \item [(a)] $p^\dagger$ is close to a point $p\in L_{\delta}\cap  L_{\delta}'\cap \pi^{-1}(U_\Lambda')$, with $T_pL_{\delta}=T_pL_{\delta}'$, \ie on $\R^n$, $q^{\dagger}:=\pi(p^\dagger)$ is close to a point $q:=\pi(p)$, away from both branch sets by distance $\Lambda^{-1}$, with $du(q)=dv(q)$ and $\nabla^2u(q)=\nabla^2v(q)$. 
       
       -- Near $q$, we shall perform a real analyticity argument (similar to Case (C) in \cite[Lemma 4.19]{IJS}) to make the perturbed Hessian difference $u-v\pm \varepsilon H$ have index $1,\cdots,n-1$, for generic $H$ defined in a neighbourhood of $q$, \textit{but with a non-vanishing gradient at $q$} (\cf \eqref{eqn:genericityonh}).  

\item [(b)] $p^\dagger$ is close to a point $p\in L_{\delta}\cap  L_{\delta}'\cap \pi^{-1}(U_\Lambda')$, with $T_pL_{\delta}\neq T_pL_{\delta}'$, \ie on $\R^n$, $q^{\dagger}:=\pi(p^\dagger)$ is close to a point $q=\pi(p)$, away from both branch sets by distance $\Lambda^{-1}$, with $du(q)=dv(q)$ but $\nabla^2u(q)\neq \nabla^2v(q)$. 
       
       -- This follows from the linear algebra fact comparing the Floer degree and the Hessian differences (\cf Lemma \ref{lem:floerdegreeforgraph}). 
    
        \item [(c)] $p^\dagger$ is close to a point $p\in L_{\delta}\cap  L_{\delta}'\cap \pi^{-1}(U_\Lambda\setminus U_{\Lambda}')$, \ie $q^{\dagger}:=\pi (p^\dagger)$ is close to a point $q:=\pi(p)$ away from the branch set $\Sigma$ (resp. $\Sigma')$ by distance $\Lambda^{-1/2}$, but in the $\Lambda^{-1}$-neighbourhood of another branch set $\Sigma'$ (resp. $\Sigma$), such that $du(q)=dv(q)$.

       -- Using the asymptote near the branch set, we can rule out this possibility. 
       
       \item [(d)] $p^\dagger $ is not close to $L_\delta\cap L_\delta'$. 
       
       -- This will not happen, as long as $0<\varepsilon\leqslant \varepsilon_0\ll 1$. 
       
   \end{itemize}

\textbf{Case (a): Real Analyticity.} 
   
    \textbf{Claim (a)}. Let $K_\Lambda:=\{q\in \bar U_{\Lambda}':du(q)=dv(q),\nabla^2u(q)=\nabla^2v(q)\}\subset \R^n$.  We choose smooth $H:U_{2\Lambda}'\to \R$ satisfying
    \begin{equation}\label{eqn:genericityonh}
        \operatorname{crit}H\cap K_{\Lambda}=\emptyset,
    \end{equation}
   to define a family of Hamiltonian perturbations on  $L_{\delta}^\dagger:=L_{\delta du+\varepsilon \delta dH}$ over $U_{\Lambda}'$. Then for generic $H$ (to achieve transversality), there exists $\delta_a=\delta(u,v,\Lambda,H)$, with $\varepsilon_a=\varepsilon_a(u,v,\delta,\Lambda,H)$, such that the item (2) in Proposition \ref{prop:perturb} holds over a sufficiently small neighbourhood $\bigcup\limits_{q\in K_\Lambda}B_{r_q}(q)$ of $K_\Lambda$. 
    
    To avoid confusion, we remark that 
    \begin{itemize}
        \item Recall that $\pi:L_\delta\to \R^n$ is a double branch cover, as we are away from the branching locus $\Sigma$, $\pi^*H$ is well-defined on $\pi^{-1}(U_{2\Lambda}')\cap L_\delta$. Such Hamiltonian perturbations $\varepsilon\delta dH=\varepsilon\delta d(\pi^*H)$ will generally break up the $\Z_2$-symmetry of $L_\delta$.  
        \item All the analyses developed below are local, meaning that working on a contractible domain contained in $\R^n\setminus\Sigma$, $u,v$ are represented by single branches $\pm u,\pm v$ respectively.   
        \item Note that the extra condition \eqref{eqn:genericityonh} is generic, as $U_\Lambda' \setminus K_{\Lambda} $ is open and dense (by the unique continuation property for $du$ and $dv$). 
    \end{itemize}

   \begin{proof}[Proof of Claim (a)]
        Around any point $q\in K_\Lambda$, we can take a local single branch of $u$ and $v$ denoted as $\pm u,\pm v$, such that $du(q)=dv(q)$ and $\nabla^2u(q)=\nabla^2v(q)$. The real-analytic Taylor theorem \cite[Theorem A.1]{IJS} (for $\nabla(u-v)$) implies for sufficiently small $r_q$, $$| \nabla (u-v)|\leqslant C_1 |q-q^{\dagger}|| \nabla^2 (u-v)|,\quad \forall q^{\dagger}\in B_{r_q}(q).$$
    
    Condition \eqref{eqn:genericityonh} implies $\nabla H(q)\neq 0$, and hence there exists $C_2>0$ such that
    $$|\nabla^2 H|\leqslant C_2|\nabla H|,\quad \forall q^\dagger\in B_{r_q}(q),$$
    shrinking $r_q$ if necessary. 

    Now suppose there is an intersection point $p^{\dagger}\in L_\delta^\dagger\cap L_\delta'\cap \pi^{-1}(\bigcup\limits_{q\in K_\Lambda}B_{r_q}(q))$. By construction,  $q^{\dagger}:=\pi(p^\dagger)\in B_{r_q}(q)$ for some $q\in K_\Lambda$. Now we aim to show that the Hessian matrix of $u-v\pm \varepsilon H$ at $q^\dagger$ is indefinite, where the non-degeneracy is achieved by the generic choice of $H$.  
    At $q^\dagger$, we have $\nabla(u-v)(q^\dagger)=\pm \varepsilon\nabla H(q^\dagger)$. Then
   \begin{equation*}
       \begin{aligned}
           \left|\operatorname{tr}\operatorname{Hess}(u-v\pm \varepsilon H)(q^\dagger)\right|&=\varepsilon \left|\operatorname{tr}\operatorname{Hess}H(q^\dagger)\right|\leqslant \varepsilon\sqrt{n}|\nabla^2H(q^\dagger)|\leqslant \sqrt{n} C_2|\nabla (\varepsilon H)(q^\dagger)|\\&=\sqrt{n}C_2|\nabla (u-v)(q^\dagger)|\leqslant \sqrt{n}C_1 C_2||q-q^{\dagger}||\nabla^2(u-v)(q^\dagger)|.
       \end{aligned}
   \end{equation*}
   On the other hand, 
   \begin{equation*}
       \begin{aligned}
           |\nabla^2(u-v\pm \varepsilon H)(q^\dagger)|&\geqslant |\nabla^2(u-v)(q^\dagger)|-\varepsilon|\nabla^2H(q^\dagger)|
           \\&\geqslant |\nabla^2(u-v)(q^\dagger)|-C_2|\nabla (\varepsilon H)(q^\dagger)|
           \\&\geqslant |\nabla^2(u-v)(q^\dagger)|-C_2|\nabla (u-v)(q^\dagger)|
           \\&\geqslant (1-C_1C_2|q-q^{\dagger}|)|\nabla^2(u-v)|(q^\dagger)|. 
       \end{aligned}
   \end{equation*}
   Once $C_1C_2|q-q^{\dagger}|\ll 1$ (true for sufficiently small $r_q$), we conclude that
   $$|\operatorname{tr}\operatorname{Hess}(u-v\pm \varepsilon H)(q^\dagger)|<\frac{1}{2}|\operatorname{Hess}(u-v\pm \varepsilon H)(q^\dagger)|,$$
   violating the standard algebraic inequality
   \begin{equation}
       \label{eqn:eigenvalueineq}
        \sum_{k=1}^n\lambda_k^2\leqslant \left(\sum_{k=1}^n\lambda_k\right)^2,\quad  \lambda_k>0,~1\leqslant k\leqslant n.
   \end{equation}
    Therefore, eigenvalues of $\operatorname{Hess}(u-v\pm \varepsilon H)(q^\dagger)$ cannot be all positive or all negative. Since points in $K_{\Lambda}$ are away from the branch sets by distance at least $\Lambda^{-1}$, the Hessians are bounded by $\Lambda^{1/2}$. The claim follows from Lemma \ref{lem:floerdegreeforgraph} and a covering argument, as $K_{\Lambda}$ is compact. 
\end{proof}
\textbf{Case (b): Generic perturbations will do. }

\textbf{Claim (b):} Let $E_{\Lambda}=\{q\in \bar U_\Lambda':du(q)=dv(q),\nabla^2u(q)\neq \nabla^2v(q)\}$. Then the Hamiltonian perturbation in Case (a), with the further requirement that $\delta\leqslant \delta_b,\varepsilon\leqslant \varepsilon_b$, will also guarantee that there are no degree $0$ or $n$ intersections in a $r_q$-ball centred at any point $q\in E_\Lambda$. 
\begin{proof}[Proof of Claim (b)]
As in the description of Case (a), we have
$$\left|\operatorname{tr}\operatorname{Hess}(u-v\pm \varepsilon H)(q^\dagger)\right|=\varepsilon\left|\operatorname{tr}\operatorname{Hess}H(q^\dagger)\right|\leqslant \varepsilon \sqrt{n}|\nabla^2H(q^\dagger)|,$$
and $$|\nabla^2(u-v\pm \varepsilon H)(q^\dagger)|\geqslant |\nabla^2(u-v)(q^\dagger)|-\varepsilon |\nabla^2H(q^\dagger)|.$$ 
Taking $r_q\ll1$ and then $\varepsilon\ll1$, such that $$|\nabla^2(u-v)|\geqslant \frac{1}{2}|\nabla^2(u-v)(q)|\geqslant2\sqrt{n}\varepsilon\sup\limits_{B_{r_q}(q)}|\nabla^2H|,\quad \textnormal{in}~~B_{r_q}(q),$$
we then conclude that 
$$|\operatorname{tr}\operatorname{Hess}(u-v\pm \varepsilon H)(q^\dagger)|<\frac{1}{2}|\operatorname{Hess}(u-v\pm \varepsilon H)(q^\dagger)|.$$
The rest is the same as Case (a), \ie taking $\delta\ll 1$ and applying Lemma \ref{lem:floerdegreeforgraph} with a covering argument (after removing a neighbourhood of $K_\Lambda$ in case (a)). 
\end{proof}
 \textbf{Case (c): No intersections when approaching a branch set but away from another.}

\textbf{Claim (c):} We have $L_\delta \cap L_{\delta}'\cap \pi^{-1}(U_\Lambda\setminus U_{\Lambda}')=\emptyset$, once $\Lambda$ is taken to satisfy $\Lambda\geqslant \Lambda_0\gg 1$. 

 \begin{proof}[Proof of Claim (c)]  
    Suppose that there exists $q\in \R^n$ such that $\rho(q)\leqslant \Lambda^{-1}$ but $\rho'(q)\geqslant \Lambda^{-1/2}$ (similarly for the dual case). Taking $\Lambda^{-1/2}\ll \frac{1}{4}\min \{\operatorname{inj}_\Sigma,\operatorname{inj}_{\Sigma'}\}$, by the structure Theorem \ref{thm:structureofz2}, we find 
    $$|\delta du|(q)\leqslant C(u)\delta\rho(q)^{1/2}\leqslant C(u)\delta\Lambda^{-1/2}.$$
    
    On the other hand, since $|dv|^{-1}(0)=\Sigma'$ and $|dv|$ has linear growth at infinity, we know by the structure Theorem \ref{thm:structureofz2} again, 
    $$|\delta dv|(q)\geqslant C(v)\delta\min \{\rho'^{1/2},1\}\geqslant C(v)\delta\Lambda^{-1/4},$$
    for some small positive number $C(v)$. Thus, $L_{\delta}\cap L'_{\delta}\cap \pi^{-1}(U_{\Lambda}\setminus U_{\Lambda}')=\emptyset$ as $\Lambda\gg 1$, which then reduces to Case (d). 
    \end{proof}

\textbf{Case (d): Continuity Argument.}
    
    \textbf{Claim (d).} For each $\delta\ll 1$ (to satisfy the requirements for Cases (a)--(c)) and $p\in \C^n\setminus (L_\delta\cap L_\delta')$, there exist $\varepsilon_p>0$ and an open neighbourhood $U_p\subset \C^n$ such that 
    $L_\delta^{\dagger}\cap L_\delta'\cap U_p=\emptyset$ for $\varepsilon<\varepsilon_p$. 
    \begin{proof}[Proof of Claim (d)] For $p\in \C^n\setminus (L_\delta\cap L_\delta')$ with further $p\notin L_\delta'$, we 
    can take a small $U_p\subset \C^n\setminus  L_\delta'$ with compact closure $\bar U_p\subset \C^n\setminus L_{\delta}'$, then
    $$L_\delta^{\dagger}\cap L_{\delta}'\cap \bar U_p\subset L_\delta'\cap \bar U_p=\emptyset.$$

   For $p\in L_\delta'\setminus L_\delta$, we know $d(p,L_\delta)>0$ as $L_\delta$ is closed. The existence of the required $U_p,\varepsilon_p$ follows from the  continuity of the Hamiltonian perturbation family $L_\delta^\dagger$. 
    \end{proof}
    Now we are able to prove Proposition \ref{prop:perturb}. Inside $(M,\omega)$, for each $\delta\leqslant \delta_0'(u,v)$ (in Proposition \ref{prop:appslag}), we take a family of Hamiltonian perturbations $L_\delta^{\dagger}$ of $L_\delta$ parametrised by $\varepsilon\in (0,\varepsilon_0')$, such that the following conditions are achieved:
    \begin{itemize}
        \item For generic $\varepsilon$, we have global transversality. 
        \item (\cf Claim (a)) The perturbation in the intermediate region $U_{2\Lambda}$ is described by a Hamiltonian perturbation pulled back from the base, \ie $H:U_{2\Lambda}\to \R$, such that the extra generic condition \eqref{eqn:genericityonh} holds. 
        \item Since $L_\delta$ is smooth away from the branch locus and is of class $C^{1,\alpha}$ around the branch set, we can arrange $\bar L_{\delta}^\dagger$ to be a \textit{smooth} family and to converge to $\bar L_\delta$ in $C^1$, with potential and grading convergence in $C^0$.   
    \end{itemize}
Now item (2) follows from a standard covering argument as follows. 

Write $ B= U_{\Lambda}\times B_R(0)\subset \C^n$, for $R\gg 1$, such that $L_\delta\cap {\pi^{-1}(U_{\Lambda})},L_{\delta}'\cap {\pi^{-1}(U_{\Lambda})}\subset B$. Since $\bar B\setminus \pi^{-1}(\bigcup\limits_{q\in K_\Lambda}B_{r_q}(q))$ is compact, a point $p$ in this compact set either satisfies
$p\in L_\delta \cap L_\delta'\cap \pi^{-1}(U_{\Lambda})$ and hence $p\in L_\delta \cap L_\delta'\cap \pi^{-1}(U_{\Lambda}')$ by Claim (c), or $p \in \C^n \setminus (L_\delta \cap L_\delta')$. Then by Claims (b) and (d), we can find finitely many small neighbourhoods $U_{p_1},\cdots,U_{p_N}$ with $\delta_{p_i}>0,\varepsilon_{p_i}>0$ such that, for $\delta<\delta_0:=\min\{\delta_a,\delta_{p_1},\cdots,\delta_{p_N}\}$ and $\varepsilon<\varepsilon_0=\min\{\varepsilon_a,\varepsilon_{p_1},\cdots,\varepsilon_{p_N}\}$,
any intersection point $p^\dagger \in L_\delta^\dagger\cap L_\delta'\cap B$ has Floer degree $\mu_{L_\delta^\dagger, L_\delta'}(p^\dagger)\in \{1,\cdots,n-1\}$. 

Since $L^{\dagger}_\delta$ and $L_\delta'$ intersect transversely in $\bar B$, they meet at finitely many points, say $p_1^{\dagger},\cdots,p_m^{\dagger}$. If $p_k^{\dagger}\in \pi^{-1}(\bigcup\limits_{q\in K_\Lambda}B_{r_q}(q))$, then it follows from Claim (a) that $\mu_{L_\delta^\dagger, L_\delta'}(p^\dagger_k)\in \{1,\cdots,n-1\}$. Otherwise, $p_k^{\dagger}\in U_{p_j}$ for some $1\leqslant j\leqslant N$. If $p_j$ belongs to Case (b), we conclude that $\mu_{L_\delta^\dagger, L_\delta'}(p^\dagger_k)\in \{1,\cdots,n-1\}$. Otherwise, $p_j$ lies in Case (d), and we derive a contradiction. This completes the proof of Proposition \ref{prop:perturb} for the non-degenerate $\Z_2$-harmonic function $v$. 

Finally, we deal with the limiting case  $L_{\delta}'=\Pi_+^\delta\cup \Pi_-^\delta$ with $\Sigma'=\emptyset$, \ie replacing the $\Z_2$-harmonic function $v$ with the asymptotic harmonic polynomial $\pm V$. Since $|\delta\nabla^2 V|\lesssim \delta$ on $\R^n$, the discussion above for the region $\rho\geqslant \Lambda^{-1}$ extends smoothly to this situation. The remaining case is

\textbf{Case (e)}. The intersection point $p^\dagger\in L_\delta^\dagger\cap L_\delta' \cap \pi^{-1}(\rho\leqslant \Lambda^{-1})$ is either close to a point $p\in L_\delta\cap L_\delta'\cap \pi^{-1}(\rho <2\Lambda^{-1})$ or not close to $L_\delta \cap L_\delta'$. 

\textbf{Claim (e)}: Choose a family of generic Hamiltonian perturbations $L_\delta^\dagger$ of $L_\delta$ as above, with the further generic condition $0\notin L_\delta^\dagger$. Then, after taking $\Lambda>\Lambda_0\gg 1$, $\delta \ll 1$, and $\varepsilon \ll1$ sequentially, any intersection point $p^\dagger \in L_\delta^{\dagger}\cap L_\delta' \cap \pi^{-1}(\{\rho<2\Lambda^{-1}\})$ has Floer degree $\mu_{L_\delta^{\dagger}, L_\delta'}(p^\dagger)\in \{1,\cdots,n-1\}$.  

The observation is that, near the branch set $\Sigma$, $TL_\delta$ contains a 2-dimensional subspace, almost vertical relative to $\R^n$ (\cf \ref{eqn:Hessunearbranch}). So in the Floer degree formula, all characteristic angles $\alpha_1,\cdots,\alpha_n$ cannot be very small simultaneously.  

\textbf{Claim (e'):} For sufficiently small $\delta\ll \Lambda^{-1/2}$, any intersection point $p\in L_\delta\cap L_\delta'\cap \pi^{-1}(\rho <2\Lambda^{-1})$, there exists a vector $0\neq W\in T_pL_\delta$ such that the angle $\theta (W,W')$ between $W$ and any non-zero vector $W'\in T_pL_{\delta}'$ lies in $[c\delta\Lambda^{1/2},\pi -c\delta\Lambda^{1/2} ]\subset (0,\pi)$, for a small number $c>0$ \textit{independent of} $\delta,\Lambda$.   

Let $q=\pi(p)$. Then $du(q)=dv(q)$. As $|du|^{-1}(0)=\Sigma$ and $|dv|^{-1}(0)=\{0\}$, either $q=0\in \Sigma$ or $q\notin \Sigma\cup \{0\}$. In the former case, one takes $W$ to belong to the $2$-dimensional subspace of $T_pL_\delta$ vertical to the base $\R^n$. Then the induced angle is approximately $\frac{\pi}{2}$ up to an error of order $\mathcal{O}(\delta)$. 

In the latter case, for $\Lambda^{-1}\ll 1$, we know that $q$ lies in a local Fermi coordinate system centred at a point on $\Sigma$. According to the local asymptotic expansion of $\nabla^2u$ (\cf \ref{eqn:Hessunearbranch}), we take $W=(I+\sqrt{-1}\delta \nabla^2u)w$ where $w\in \R^n$ with $|w|=1$ is an eigenvector of $N$ with eigenvalue $\frac{3}{4}\rho^{-1/2}|{\bf{B}}|$.  

Then the induced angle $\theta(W,W')\in [0,\pi]$ between $W$ and any vector $W'=(I+\sqrt{-1}\delta \nabla^2v)(w')\in T_pL_{\delta'}=(I+\sqrt{-1}\delta \nabla^2v)\R^n$ for a unit vector $w'\in \R^n$ is computed by the cosine formula:
\begin{equation*}
    \begin{aligned}
        |\cos\theta(W,W')|&=\frac{|\langle W,W'\rangle_{\C^n}|}{|W||W'|}=\frac{|\langle w,w'\rangle_{\R^n}+\langle \delta\nabla^2u(w), \delta\nabla^2v(w')\rangle|}{\sqrt{1+\delta^2|\nabla^2 u(w)|^2}\sqrt{1+\delta^2|\nabla^2v(w')|^2}}\\&\leqslant \frac{1+|\delta\nabla^2u(w)||\delta\nabla^2v(w')|}{\sqrt{1+\delta^2|\nabla^2 u(w)|^2}\sqrt{1+\delta^2|\nabla^2v(w')|^2}}\\&=\cos(\arctan(|\delta\nabla^2u(w)|)-\arctan (|\delta \nabla^2v(w')|)).
    \end{aligned}
\end{equation*}
Thus, using \eqref{eqn:Hessunearbranch} and taking $\delta \ll \Lambda^{-1/2}$, we find

\begin{equation*}
    \begin{aligned}
        \min \{|\theta(W,W')|,|\pi-\theta(W,W')|\}&\geqslant \arctan(|\delta\nabla^2u(w)|)-\arctan (|\delta \nabla^2v(w')|)
        \\& \geqslant \arctan\left(\delta \left(\frac{3}{4}\rho^{-1/2}|{\bf{B}}|+\mathcal{O}(\rho^{\alpha})\right)\right)-\arctan(C\delta) 
        \\& \geqslant c\delta \Lambda^{1/2}, 
    \end{aligned}
\end{equation*}
where $c=c(u,v)>0$ depends on the non-degeneracy of $u$ and the uniform Hessian bound for $|\nabla^2 v|$. Claim (e') follows. 
\begin{proof}[Proof of Claim (e)]
   Let $p^\dagger\in L_\delta^\dagger\cap L_\delta'\cap \pi^{-1}(\rho\leqslant \Lambda^{-1})$. Then, up to a very similar covering argument, we can assume $p^\dagger$ lies in a very small neighbourhood of $p\in L_\delta\cap L_\delta'\cap \pi^{-1}(\rho <2\Lambda^{-1})$ (otherwise $p$ lies in a non-intersection neighbourhood, and we are done by Case (d)). We identify $T_{p^\dagger}L_\delta^\dagger\cong \R^n$ and $T_{p^{\dagger}}L_{\delta}'\cong (e^{i\alpha_1},\cdots,e^{i\alpha_n})\R^n$ with $\alpha_1,\cdots,\alpha_n\in (0,\pi)$. As the Hamiltonian perturbation can be made arbitrarily small in $C^1$-norm, Claim (e') implies that at least one $\alpha_k$ ($1\leqslant k\leqslant n$) satisfies
    $$\min\{|\alpha_k|,|\pi-\alpha_k|\}\geqslant c\delta \Lambda^{1/2}+\Psi(\varepsilon|u,v,\delta,\Lambda).$$
    From Proposition \ref{prop:appslag}, we know $\theta_{L_\delta^\dagger}=\theta_{L_\delta}+\Psi(\varepsilon|u,v,\delta,\Lambda)=\mathcal{O}(\delta^{1+\alpha})+\Psi(\varepsilon|u,v,\delta,\Lambda)$, and $\theta_{L_\delta'}=\mathcal{O}(\delta^3)$. Substituting into the Floer degree formula, after taking $\Lambda\gg 1$, $\delta \ll\Lambda^{-1/2}\ll1 $, and $\varepsilon \ll1$ sequentially, it follows that
    
    $$\mu_{L_\delta^{\dagger},L_{\delta}'}(p^\dagger)=\frac{\sum\limits_{k=1}^n\alpha_k+\theta_{L_\delta^\dagger}(p^\dagger)-\theta_{L_\delta'}(p^\dagger)}{\pi}\left\{\begin{array}{l}
        \geqslant \frac{c\delta \Lambda^{1/2}+\mathcal{O}(\delta^{1+\alpha})+\Psi(\varepsilon|u,v,\delta,\Lambda)}{\pi}>0,   \\
        \leqslant \frac{n\pi -c\delta\Lambda^{1/2}+\mathcal{O}(\delta^{1+\alpha})+\Psi(\varepsilon|u,v,\delta,\Lambda)}{\pi}<n, 
    \end{array}\right.$$
Since the Floer degree $\mu_{L_\delta^{\dagger},L_{\delta}'}$ is an integer, Claim (e) follows. 
\end{proof}
Proposition \ref{prop:perturb} for the case $L_\delta'=\Pi_-^\delta\cup \Pi_+^\delta$ follows from the previous arguments and Claim (e). 
\end{proof}

\begin{cor}\label{cor:M0>0}
    We have $M_0>0$. 
\end{cor}

\begin{proof}
   According to the construction of the $\delta$-dependent partial compactification $(M,\omega)$ (\cf Section \ref{sec:partialcompact}) and Proposition \ref{prop:perturb}, we can find a family of Hamiltonian perturbations $\bar L_\delta^\dagger$ of $L_\delta$ parametrised by $\varepsilon\ll 1$, which intersects $\bar L_\delta'=\overline{\Pi_-^\delta\cup\Pi_+^\delta}=S_-^n\cup S_+^n$ transversely and has no degree $0$ or $n$ intersections over $B_\Lambda(0)$. Then by \cite[Theorem 2.17]{IJS}, 
    we have a distinguished triangle 
    $$S_+^n[-1]\to \bar L_\delta^{\dagger}\to S_-^n\to S_+^n. $$
    Then \cite[Theorem 2.16]{IJS} implies that there exists a holomorphic triangle $\Delta$ whose corners are $p\in S_{+}^n[-1]\cap \bar L_\delta^{\dagger}$, $q\in \bar L_\delta^{\dagger}\cap S_-^n$, and $r\in  S_-^n\cap S_+^n[-1]=\{0\}$. These corners have Floer degrees $\mu_{S_{+}^n[-1], \bar L_\delta^{\dagger}}(p)=0$, $\mu_{\bar L_\delta^{\dagger}, S_-^n}(q)=0$, and $\mu_{S_-^n,S_{+}^n[-1]}(r)=1$. Note that $S_-^n\cup S_+^n[-1]$ has the same Lagrangian angle as the union of two planes $\bar L_\delta'$. After taking $\Lambda \geqslant \Lambda_0\gg 1$, we know that $p,q\notin \pi^{-1}(B_{\Lambda}(0))$ by the perturbation result and that they belong to \textit{different} ends of $\bar L_{\delta}'$. Using Stokes' formula, Remark \eqref{rmk:potentialinv}, \eqref{eqn:Lagpotential} and \eqref{eqn:asym-of-u}, we compute the symplectic area of $\Delta$, 
    \begin{equation*}
        \begin{aligned}
            0<\operatorname{Area}(\Delta)=\int_{\partial \Delta}\tilde{\lambda}&=\int_{0\to q}df_{S_-^n}+\int_{q\to p}df_{\bar L_\delta^\dagger}+\int_{p\to 0}df_{S^n_+[-1]}
            \\&=f_{S_-^n}(q)-f_{S_-^n}(0)+f_{\bar L_\delta^\dagger}(p)-f_{\bar L_\delta^\dagger}(q)+f_{S_+^n}(0)-f_{S_+^n}(p)\quad
            \\&=f_{\bar L_\delta^\dagger}(p)-f_{\bar L_\delta^\dagger}(q)
            \\&=f_{\bar L_\delta}(p)-f_{\bar L_\delta}(q)+\Psi(\varepsilon|\delta,u,v,\Lambda)
            \\&= 2\delta M_0+\mathcal{O}(\delta \Lambda^\mu)+\Psi(\varepsilon|\delta,u,v,\Lambda).
        \end{aligned}
    \end{equation*}
Taking $\varepsilon \to 0^+$, we get $M_0+\mathcal{O}(\Lambda^\mu)\geqslant 0$ for all $\Lambda \gg 1$. Thus, $M_0\geqslant 0$. 

Suppose that $M_0=0$. Then by \eqref{eqn:higher-order derivative estimates} and the structure Theorem \ref{thm:structureofz2}, 
$f_{L_{du}}=-u+\frac{1}{2}\langle x,\nabla u\rangle$ is a $\Z_2$-harmonic function with asymptotic value $0$ as one approaches the branching set $\Sigma$ or the infinity. Applying the maximum principle to the subharmonic function $|f_{L_{du}}|^2$ on $B_{\Lambda}(0)\setminus B_{\Lambda^{-1}}(\Sigma)$ and letting $\Lambda \to \infty$, we obtain $f_{L_{du}}\equiv 0$, meaning that $u$ is homogeneous with respect to $0\in \R^n$. Then $u=\pm V$, contradicting our assumption.   
\end{proof}

Combining Corollary \ref{cor:indexv}, \ref{cor:M0>0} with Proposition \ref{prop: para-match}, we can find a unique candidate $\Z_2$-harmonic function $v$ arising from the Lawlor necks, whose asymptotic harmonic polynomial at infinity is the same as $u$. For $\delta\ll 1$, $L_\delta'=\operatorname{graph}(\delta dv)$ defines an approximate special Lagrangian (\cf Proposition \ref{prop:appslag}). 

Suppose for contradiction that $L_{\delta }\neq L_{\delta }'$, \ie they do not coincide with each other on an open set. Applying the perturbation result to their partial compactifications inside $(M,\omega)$ (depending on $\delta$), we get a family of Hamiltonian perturbations
$\bar L_{\delta}^{\dagger}$ which intersects $\bar L_{\delta}'$ transversely and has no degree $0$ or $n$ intersections over $B_{\Lambda}(0)\setminus S_{\Lambda}$. 

\begin{cor}
\label{cor:potentialdiffsmallest}
    For sufficiently small $\varepsilon\ll 1$, at any intersection point $p\in \bar L_\delta^\dagger\cap \bar L_\delta'$ of degree $0$ or $n$, the potential difference is small,
    \begin{equation*}
        |f_{\bar L_\delta^{\dagger}}-f_{\bar L_\delta'}|(p)\leqslant C\delta \Lambda^{-1/4}+\Psi(\varepsilon|\delta,u,v,\Lambda).
    \end{equation*}
    
\end{cor}
\begin{proof}
    Note that all possible intersections of Floer degree $0,n$ either lie in (after projecting to the base) the $\Lambda^{-1/2}$-neighbourhood to both the branching sets or the complement of $\Lambda$-neighbourhood of $0\in \R^n$. Therefore, the estimate on potential differences follows from Lemma \ref{lem:reg+potential}, the structure Theorem \ref{thm:structureofz2} and the choice of $v$. 
\end{proof}

\subsection{Small holomorphic disk argument}
    Now we are ready to get a contradiction by showing that $L_\delta$ and $L_\delta'$ coincide away from the branching sets: 

    \begin{prop}\label{prop:uvagreeatq0}
        For any $q_0\in \R^n$ such that $\rho(q_0),\rho'(q_0)\geqslant 1$, we have $du(q_0)=dv(q_0)$ (up to sign).  
    \end{prop}
    
     As $q_0$ is away from both branching sets, over a small neighbourhood of $q_0$, $du,dv$ separate into the union of two single-valued branches respectively, and 
    all Lagrangians $L_\delta,L_\delta'$ split into two sheets, which are topological disks with $C^{1,\alpha}$-bounded geometry (as the Hamiltonian perturbation can be arbitrarily $C^\infty$-small away from the branching set). 

    Suppose to the contrary that $du(q_0)\neq dv(q_0)$ (the case $du(q_0)\neq -dv(q_0)$ follows from replacing $v$ by $-v$). By the above discussion, we can find $r_0\leqslant 1$ (of order $\mathcal{O}(1)$, compared with $\delta$ below) and 
    $$C_0=:\min \{|du-dv|(q_0),|du+dv|(q_0),|du|(q_0),|dv|(q_0)\}>0,$$
    such that 
    \begin{itemize}
        \item $|du(q_1)\pm dv(q_2)|\geqslant \frac{3}{4}C_0$ holds for all $q_1,q_2\in B_{2r_0}(q_0)$. 
        \item $|du|(q),|dv|(q)\geqslant \frac{3}{4}C_0$ for all $q\in B_{2r_0}(q_0)$.  
    \end{itemize}Then for $\delta \leqslant \delta_0$ and sufficiently small $\varepsilon\leqslant \varepsilon_0(\delta,u,v,\Lambda,q_0)\ll1$, we have
    \begin{equation}
        \label{eqn:distanceoftwolagsnearq0}
        d_{\C^n}(L_\delta^\dagger\cap \pi^{-1}(B_{r_0}(q_0)),L_\delta'\cap \pi^{-1}(B_{r_0}(q_0)))\geqslant \frac{\delta}{2}C_0,
    \end{equation}
    and over $B_{r_0}(q_0)$, $L_\delta^{\dagger}$ also splits into two sheets separated by distance $\frac{\delta}{2}C_0$.  
    
     Set $p_0'=(q_0,\delta dv(q_0))\in L_{\delta}'$. We arrange the auxiliary parameter $T=T(\delta)\gg 1$ to be sufficiently large appearing in the partial compactification procedure (\cf Section \ref{sec:partialcompact}), such that on a large ball in $\C^n$ containing $p_0'$, the symplectic Calabi--Yau structure agrees with the standard one on $\C^n$, for all $ 0<\delta\ll 1$.  
    
    By Theorem \ref{thm:classifyaclag},
    $\bar L_\delta^{\dagger}$ and $\bar L_\delta'$ are isomorphic in $D^bFuk(T^*S^n\#T^*S^n)$. It follows from Proposition \ref{prop:holomorphicstrip} and Corollary \ref{cor:potentialdiffsmallest} that
    
\begin{lem}
    There exists a (possibly broken) holomorphic strip $S$, passing through $p_0'$, with boundary on $\bar L_\delta^\dagger$ and $\bar L_\delta'$ and with symplectic area
    $\int_{S}\omega\leqslant C\delta \Lambda^{-1/4}+\Psi(\varepsilon|\delta,u,v,\Lambda).$
\end{lem}
    To derive a contradiction, we aim to show that the symplectic area 
is bounded from below by $\eta\delta$, for a small number $\eta\ll 1$ depending on the local geometry around $q_0$, but independent of $\delta$ (\cf Figure \ref{fig:area}). The key input is the following dichotomy using a local monotonicity formula for the holomorphic strip $S$ around $p_0'$.
\begin{figure}[H]
    \centering
    \includegraphics[width=0.6\linewidth]{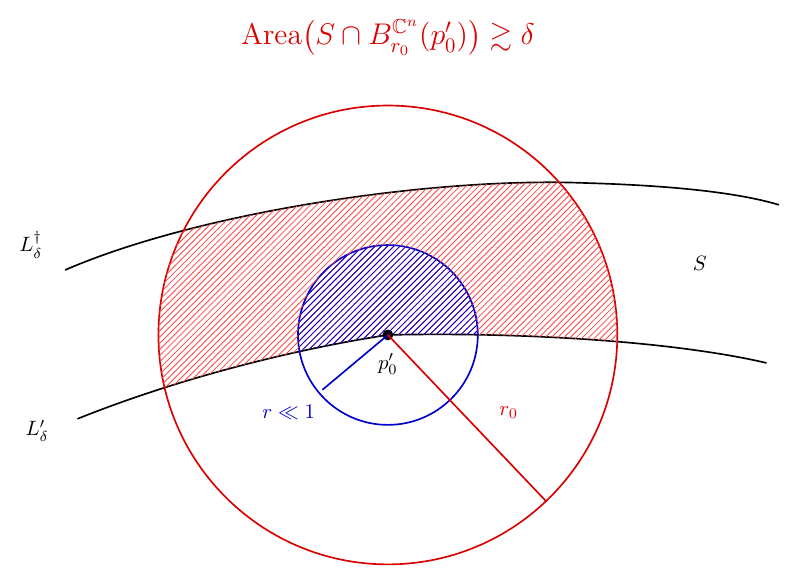}
    \caption{The heuristic of the area lower bound}
    \label{fig:area}
\end{figure}
\begin{prop}
    There exists small $\eta_0=\eta_0(u,v,r_0,C_0)$ independent of $\delta$ such that for a.e. $r$ with $0< r\leqslant r_0$ and $\eta\leqslant \eta_0$, we have the following alternatives:
    \begin{enumerate}
        \item [\textnormal{(1)}] either the length of $\Gamma_r$ is bounded from below by $\eta \delta$, \ie $\mathcal{H}^1(\Gamma_r)\geqslant \eta \delta$, where $\Gamma_r:=S\cap \partial B_r^{\C^n}(p_0')$;
        
        \item [\textnormal{(2)}] or the local area is small $\operatorname{Area}(S\cap B_r^{\C^n}(p_0'))=\int_{S\cap B_r^{\C^n}(p_0')}\omega\leqslant C  (\eta\delta)^2$, where $C=C(u,v,r_0,C_0)$ is independent of $\eta,\delta$.  
    \end{enumerate}
\end{prop}
\begin{proof}
    Let $\rho_{p_0'}(\cdot)=d_{\C^n}(\cdot,p_0')$ be the distance function to $p_0'\in \C^n$. Then for a.e. $r>0$, $\Gamma_r$ is a compact $1$-dimensional manifold with boundary, and hence consists of finitely many loops, and arcs with endpoints on $\partial S\subset L_\delta^{\dagger}\cup L_\delta'$.  

    We call $\gamma\subset \Gamma_r$ a \textit{bridge arc} if its two endpoints lie on $\partial S\cap L_{\delta}^{\dagger}$ and $\partial S\cap L_\delta'$. Then \eqref{eqn:distanceoftwolagsnearq0} implies that $\operatorname{Length}(\gamma)\geqslant \frac{\delta}{2}C_0$. We choose $\eta_0\ll C_0$ to guarantee that there are no bridge arcs within $\Gamma_r$. 

In addition to bridge arcs, we also exclude arcs whose endpoints lie on
two different local sheets of the same Lagrangian. As 
two local sheets of \(L_\delta^{\dagger}\) and of \(L'_\delta\) are separated by
\(\frac{C_0}{2}\delta\), by arranging
\(\eta_0\ll C_0\), the condition \(\mathcal H^1(\Gamma_r)<\eta\delta\)
also rules out such sheet-switching arcs.
    
    Below we perform a filling argument to show the second alternative, assuming the first alternative fails for this $r$, \ie $\mathcal{H}^1(\Gamma_r)<\eta \delta$ (\cf Figure \ref{fig:diskfilling}). 

    For any arc $\gamma \subset \Gamma_r$, its two endpoints must lie on the same Lagrangian $L_\delta^\dagger$ or $L_\delta'$. Without loss of generality, we assume the former case. Then there exists another arc $\sigma_\gamma\subset L_\delta^\dagger\cap {\pi^{-1}(B_{2r_0}(q_0))}$ joining the endpoints of $\gamma$, with 
    $\operatorname{Length}(\sigma_\gamma)\lesssim \operatorname{Length}(\gamma)$. If $\gamma \subset \Gamma_r$ is a closed loop, we then set $\sigma_\gamma=\emptyset$.

    \begin{figure}[H]
        \centering
        \includegraphics[width=0.7\linewidth]{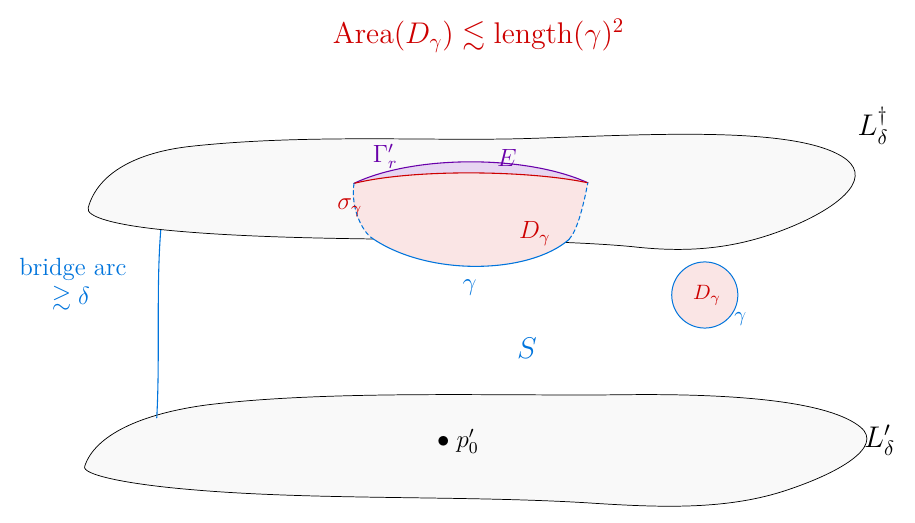}
        \caption{Disk filling with area bound}
        \label{fig:diskfilling}
    \end{figure}
    
    In all cases, the Euclidean isoperimetric inequality implies that we can find a \(2\)-chain \(D_\gamma\subset \C^n\) with boundary $\partial D_\gamma=\gamma-\sigma_\gamma$ and area bound
    $$\operatorname{Area}(D_\gamma)\lesssim \operatorname{Length}(\gamma)^2.$$
    
    The same filling argument applies to those $\gamma$ with endpoints on $L_\delta'$. 

    Set $\partial(S\cap B_r^{\C^n}(p_0'))=\Gamma_r\cup \Gamma_r'$, where $\Gamma_r'=\partial S\cap B_{r}^{\C^n}(p_0')\subset L_\delta^{\dagger} \cup L_\delta'$. Then we have  
    $$\partial\left(S\cap B_{r}^{\C^n}
    (p_0')-\sum_{\gamma} D_\gamma\right)=\Gamma_r'+\sum_{\gamma\subset \Gamma_r}\sigma_\gamma\subset  L_\delta^\dagger\cup L_\delta'$$
Moreover, since there are no bridge arcs, this remaining boundary splits as a sum of two closed \(1\)-chains, one contained in \(L_\delta^\dagger\) and the other in \(L_\delta'\). These local Lagrangian pieces are topological disks, so \(\Gamma_r'+\sum\limits_{\gamma\subset\Gamma_r}\sigma_\gamma\) bounds a \(2\)-chain
$
    E\subset L_\delta^\dagger\cup L_\delta'$. 
Then $S\cap B_r^{\C^n}(p_0')-\sum\limits_{\gamma\subset \Gamma_r} D_\gamma-E$
is a closed \(2\)-cycle with $\int_E\omega=0$ by the Lagrangian condition. 
   
As $\omega=d\lambda$ is exact, it follows that
    \begin{equation*}
        \begin{aligned}
            \int_{S\cap B_r^{\C^n}(p_0')}\omega&=\sum_{\gamma\subset \Gamma_r}\int_{D_\gamma}\omega+\int_E \omega\\&\leqslant \sum_{\gamma\subset \Gamma_r}\operatorname{Area}(D_\gamma)\lesssim\sum_{\gamma\subset \Gamma_r}\operatorname{Length}(\gamma)^2\\&\leqslant  \left(\sum_{\gamma\subset \Gamma_r}\operatorname{Length}(\gamma)\right)^2=\mathcal{H}^1(\Gamma_r)^2\leqslant (\eta\delta)^2,
        \end{aligned}
    \end{equation*}
    completing the proof of the second alternative.
\end{proof}

\begin{lem}
    For a.e. $r\geqslant \frac{\delta}{4}C_0$, for $\eta\leqslant\eta_0\ll 1$ independent of $\delta$, we have $\mathcal{H}^1(\Gamma_r)\geqslant \eta\delta$.  
\end{lem}
\begin{proof}
    Note that by \eqref{eqn:distanceoftwolagsnearq0}, inside $B_{\frac{\delta}{4}C_0}^{\C^n}(p_0')$ equipped with standard symplectic form $\omega$ and complex structure $J$, the boundary of $S$ lies only on $L_\delta'$. The condition in the monotonicity formula (\cf \cite[Lemma 45]{LiS}) is satisfied, and consequently,
     \begin{equation}
         \label{eqn:arealowerbound}
         \operatorname{Area}(S\cap B_{\frac{\delta}{4}C_0}^{\C^n}(p_0'))=\int_{S\cap B_{\frac{\delta}{4}C_0}^{\C^n}(p_0')}\omega\geqslant C\left(\frac{\delta}{4}C_0\right)^2=C\delta^2.
     \end{equation} By choosing $\eta_0\ll 1$ such that the second alternative is violated for $r\geqslant \frac{\delta}{4}C_0$ (note that the area function is non-decreasing in $r$), the lemma follows. 
\end{proof}
Fixing a suitable $\eta$, we have
\begin{cor}\label{cor:arealowerbound}
    For $\delta\ll 1$, we have $\operatorname{Area}(S\cap B_{r_0}^{\C^n}(p_0'))\geqslant C\eta  r_0 \delta$. 
\end{cor}
\begin{proof}
    The function $f(r):=\operatorname{Area}(S\cap B_{r}^{\C^n}(p_0'))$ is non-decreasing and absolutely continuous in $r$. The co-area formula, applied to $\rho_{p'_0}=dist(\cdot, p_0')$, implies
    $$f'(r)=\int_{S\cap \partial B_r^{\C^n}(p_0')}\frac{1}{|\nabla^S\rho_{p_0'}|}d\ell\geqslant \mathcal{H}^1(\Gamma_r)\geqslant \eta \delta,\quad \textit{for a.e. } r\geqslant \frac{\delta}{4}C_0.$$ 
    Combining this with \eqref{eqn:arealowerbound} and integrating over $[\frac{\delta}{4}C_0,r_0]$, we conclude
    $$f(r_0)\geqslant f\left(\frac{\delta}{4}C_0\right)+\int_{\frac{\delta}{4}C_0}^{r_0}f'\geqslant C\delta^2+\eta\delta\left(r_0-\frac{\delta}{4}C_0\right)\geqslant C\eta  r_0 \delta,$$
    for $\delta\ll 1$. 
\end{proof}
\begin{proof}[Proof of Proposition \ref{prop:uvagreeatq0}]
    From Corollary \ref{cor:potentialdiffsmallest} and \ref{cor:arealowerbound}, we know

    $$C\eta  r_0 \delta\leqslant\operatorname{Area}(S\cap B_{r_0}^{\C^n}(p_0'))\leqslant \operatorname{Area}(S)\leqslant C\delta \Lambda^{-1/4}+\Psi(\varepsilon|\delta,u,v,\Lambda),$$
    which is a contradiction once we sequentially take sufficiently large $\Lambda\gg 1$ and $\delta\ll 1$ as above, and let $\varepsilon\to 0^+$. 
\end{proof}
Since the choice of $q$ is arbitrary in Proposition \ref{prop:uvagreeatq0}, we get a contradiction to the hypothesis that $L_\delta $ and $L_\delta'$ do not agree with each other on an open set. Then the main theorem \ref{mainthm} follows from the unique continuation property. 
\begin{cor}[Theorem \ref{mainthm}]
    We have $u=v$ as $\Z_2$-harmonic functions. 
\end{cor}
\begin{proof}
   The unique continuation property (on each contractible ball) implies $du$ and $dv$ agree away from both branching sets $\Sigma\cup \Sigma'$. As they are continuous by structure Theorem \ref{thm:structureofz2}, it follows that 
   $\Sigma=|du|^{-1}(0)=|dv|^{-1}(0)=\Sigma'$. After choosing a single branch at infinity of $\R^n$, we know $u-v \sim  \mathcal{O}(|x|^\mu)$. A simple connectedness argument implies $u \equiv v$, completing the proof.
\end{proof}
We end with listing two interesting related questions,
\begin{itemize}
    \item Imagi--Joyce--Oliveira dos Santos also proves the uniqueness for the self-expanders of the Lagrangian mean curvature flow constructed in \cite{Joyce-Lee-Tsui}. In our situation, one can extract families of $\Z_2$-functions satisfying the shifted Laplacian equation
    $$\Delta f+\alpha(x\cdot\nabla f-2f)=0,\quad \alpha> 0.$$
    Then $u(x,t):=2\alpha tf\left(\frac{x}{\sqrt{2\alpha t}}\right)$ solves the heat equation $\partial_t u=\Delta u$. It would be interesting to generalise Theorem \ref{mainthm} to these families of self-similar solutions to the ``$\Z_2$-heat equation" using a very similar technique and study the relation with singularity formations of ``$\Z_2$-flows", although such a parabolic theory has not been developed yet.
    \item It is natural to ask what happens when the asymptotic polynomial $V$ is allowed to be degenerate. The simplest case is when $\partial_{x_k}V=0$ for some coordinate direction in $\R^n$, then a maximum principle argument implies that $|\partial_k u|\equiv0$, and the problem is reduced to a lower dimension. 
\end{itemize}

\bibliographystyle{alpha}
\bibliography{ref}

@article{Lockhart-McOwen,
  author       = {Lockhart, Robert B. and McOwen, Robert C.},
  title        = {Elliptic differential operators on noncompact manifolds},
  journal      = {Annali della Scuola Normale Superiore di Pisa. Classe di Scienze},
  series       = {Serie IV},
  volume       = {12},
  number       = {3},
  pages        = {409--447},
  year         = {1985},
  url          = {https://www.numdam.org/item/ASNSP_1985_4_12_3_409_0/}
}

@book{harvey1990spinors,
  title={Spinors and calibrations},
  author={Harvey, F Reese},
  volume={8},
  year={1990},
  publisher={Elsevier}
}

@article{haydys2015compactness,
  title={A compactness theorem for the Seiberg--Witten equation with multiple spinors in dimension three},
  author={Haydys, Andriy and Walpuski, Thomas},
  journal={Geometric and Functional Analysis},
  volume={25},
  number={6},
  pages={1799--1821},
  year={2015},
  publisher={Springer}
}

@book{DeLellis-Spadaro,
  author       = {De Lellis, Camillo and Spadaro, Emanuele Nunzio},
  title        = {{$Q$}-valued functions revisited},
  series       = {Memoirs of the American Mathematical Society},
  volume       = {211},
  number       = {991},
  publisher    = {American Mathematical Society},
  address      = {Providence, RI},
  pages        = {vi+79},
  year         = {2011},
  doi          = {10.1090/S0065-9266-10-00607-1},
  eprint       = {0803.0060},
  archivePrefix= {arXiv},
  primaryClass = {math.AP}
}

@article{Donaldson-Deformation,
  author       = {Donaldson, Simon K.},
  title        = {Deformations of multivalued harmonic functions},
  journal      = {The Quarterly Journal of Mathematics},
  volume       = {72},
  number       = {1--2},
  pages        = {199--235},
  year         = {2021},
  doi          = {10.1093/qmath/haab018},
  eprint       = {1912.08274},
  archivePrefix= {arXiv},
  primaryClass = {math.DG}
}

@misc{Donaldson-Twistor,
  author       = {Donaldson, Simon K.},
  title        = {Twistor construction of some multivalued harmonic functions on {${\bf R}^{3}$}},
  year         = {2025},
  eprint       = {2504.12716},
  archivePrefix= {arXiv},
  primaryClass = {math.DG},
  note         = {arXiv:2504.12716}
}

@incollection{Donaldson-Adiabatic,
  author       = {Donaldson, Simon K.},
  title        = {Adiabatic limits of co-associative {Kovalev--Lefschetz} fibrations},
  booktitle    = {Algebra, Geometry, and Physics in the 21st Century},
  series       = {Progress in Mathematics},
  volume       = {324},
  pages        = {1--29},
  publisher    = {Birkh{\"a}user/Springer},
  year         = {2017},
  doi          = {10.1007/978-3-319-59939-7_1},
  eprint       = {1603.08391},
  archivePrefix= {arXiv},
  primaryClass = {math.DG}
}

@misc{Yan-Construction,
  author       = {Yan, Dashen},
  title        = {A construction of non-degenerate {$\mathbb{Z}_{2}$}-harmonic functions on {$\mathbb{R}^{n}$}},
  year         = {2025},
  eprint       = {2503.19286},
  archivePrefix= {arXiv},
  primaryClass = {math.DG},
  note         = {arXiv:2503.19286}
}

@misc{Sun-Pell,
  author       = {Sun, Weifeng},
  title        = {On {$Z_2$} harmonic functions on {$\mathbb{R}^2$} and the polynomial {Pell}'s equation},
  year         = {2022},
  eprint       = {2209.11893},
  archivePrefix= {arXiv},
  primaryClass = {math.DG},
  note         = {arXiv:2209.11893}
}

@article{Haydys-Mazzeo-Takahashi-Examples,
  author       = {Haydys, Andriy and Mazzeo, Rafe and Takahashi, Ryosuke},
  title        = {New examples of {$\mathbb{Z}/{2}$}-harmonic {$1$}-forms and their deformations},
  journal      = {Geometriae Dedicata},
  volume       = {219},
  number       = {2},
  pages        = {30},
  year         = {2025},
  doi          = {10.1007/s10711-025-00992-w},
  eprint       = {2307.06227},
  archivePrefix= {arXiv},
  primaryClass = {math.DG}
}

@misc{Haydys-Mazzeo-Takahashi-IndexGraph,
  author       = {Haydys, Andriy and Mazzeo, Rafe and Takahashi, Ryosuke},
  title        = {An index theorem for {$\mathbb{Z}/2$}-harmonic spinors branching along a graph},
  year         = {2023},
  eprint       = {2310.15295},
  archivePrefix= {arXiv},
  primaryClass = {math.DG},
  note         = {arXiv:2310.15295}
}

@article{IJS,
  author       = {Imagi, Yohsuke and Joyce, Dominic and Oliveira dos Santos, Joana},
  title        = {Uniqueness results for special {Lagrangians} and {Lagrangian} mean curvature flow expanders in {${\mathbb C}^{m}$}},
  journal      = {Duke Mathematical Journal},
  volume       = {165},
  number       = {5},
  pages        = {847--933},
  year         = {2016},
  doi          = {10.1215/00127094-3167275},
  eprint       = {1404.0271},
  archivePrefix= {arXiv},
  primaryClass = {math.DG}
}

@article{Abouzaid-Smith,
  author       = {Abouzaid, Mohammed and Smith, Ivan},
  title        = {Exact {Lagrangians} in plumbings},
  journal      = {Geometric and Functional Analysis},
  volume       = {22},
  number       = {4},
  pages        = {785--831},
  year         = {2012},
  doi          = {10.1007/s00039-012-0162-y},
  eprint       = {1107.0129},
  archivePrefix= {arXiv},
  primaryClass = {math.SG}
}

@article{Lawlor,
  author       = {Lawlor, Gary},
  title        = {The angle criterion},
  journal      = {Inventiones Mathematicae},
  volume       = {95},
  number       = {2},
  pages        = {437--446},
  year         = {1989},
  doi          = {10.1007/BF01393905}
}

@misc{LiS,
  author       = {Li, Yang and Sz{\'e}kelyhidi, G{\'a}bor},
  title        = {Singularity formations in {Lagrangian} mean curvature flow},
  year         = {2024},
  eprint       = {2410.22172},
  archivePrefix= {arXiv},
  primaryClass = {math.DG},
  note         = {arXiv:2410.22172}
}

@article{He-BranchedSLag,
  author       = {He, Siqi},
  title        = {The branched deformations of the special {Lagrangian} submanifolds},
  journal      = {Geometric and Functional Analysis},
  volume       = {33},
  pages        = {1266--1321},
  year         = {2023},
  doi          = {10.1007/s00039-023-00645-8},
  eprint       = {2202.12282},
  archivePrefix= {arXiv},
  primaryClass = {math.DG},
}

@article{Joyce-Lee-Tsui,
  author       = {Joyce, Dominic and Lee, Yng-Ing and Tsui, Mao-Pei},
  title        = {Self-similar solutions and translating solitons for {Lagrangian} mean curvature flow},
  journal      = {Journal of Differential Geometry},
  volume       = {84},
  number       = {1},
  pages        = {127--161},
  year         = {2010},
  doi          = {10.4310/jdg/1271271795},
  eprint       = {0801.3721},
  archivePrefix= {arXiv},
  primaryClass = {math.DG}
}

@article{Takahashi-Moduli,
  author       = {Takahashi, Ryosuke},
  title        = {The moduli space of {$S^1$}-type zero loci for {$\mathbb{Z}/2$}-harmonic spinors in dimension 3},
  journal      = {Communications in Analysis and Geometry},
  volume       = {31},
  number       = {1},
  pages        = {119--242},
  year         = {2023},
  doi          = {10.4310/CAG.2023.V31.N1.A5},
  eprint       = {1503.00767},
  archivePrefix= {arXiv},
  primaryClass = {math.DG},
}

@article{Takahashi-Index,
  author       = {Takahashi, Ryosuke},
  title        = {Index theorem for {$\mathbb{Z}/{2}$}-harmonic spinors},
  journal      = {Mathematical Research Letters},
  volume       = {25},
  number       = {5},
  pages        = {1645--1671},
  year         = {2018},
  doi          = {10.4310/MRL.2018.v25.n5.a13},
  eprint       = {1705.01954},
  archivePrefix= {arXiv},
  primaryClass = {math.DG}
}

@article{Parker-Deformation,
  author       = {Parker, Gregory J.},
  title        = {Deformations of {$\mathbb Z_2$}-harmonic spinors on {$3$}-manifolds},
  journal      = {Geometric and Functional Analysis},
  volume       = {36},
  number       = {1},
  pages        = {201--300},
  year         = {2026},
  doi          = {10.1007/s00039-026-00729-1},
  eprint       = {2301.06245},
  archivePrefix= {arXiv},
  primaryClass = {math.DG}}

@article{Doan-Walpuski,
  author       = {Doan, Aleksander and Walpuski, Thomas},
  title        = {On the existence of harmonic {$\mathbb{Z}_{2}$} spinors},
  journal      = {Journal of Differential Geometry},
  volume       = {117},
  number       = {3},
  pages        = {395--449},
  year         = {2021},
  doi          = {10.4310/jdg/1615487003},
  eprint       = {1710.06781},
  archivePrefix= {arXiv},
  primaryClass = {math.DG}}

@misc{He-Parker,
  author       = {He, Siqi and Parker, Gregory J.},
  title        = {{$\mathbb Z_2$}-harmonic spinors and {$1$}-forms on connected sums and torus sums of {$3$}-manifolds},
  year         = {2024},
  eprint       = {2407.10922},
  archivePrefix= {arXiv},
  primaryClass = {math.DG},
  note         = {arXiv:2407.10922}
}

@misc{Salm,
  author       = {Salm, Willem Adriaan},
  title        = {Construction of {$\mathbb{Z}_2$}-harmonic {$1$}-forms on closed {$3$}-manifolds with long cylindrical necks},
  year         = {2024},
  eprint       = {2410.07015},
  archivePrefix= {arXiv},
  primaryClass = {math.DG},
  note         = {arXiv:2410.07015}
}

@article{Taubes-PSL2C,
  author       = {Taubes, Clifford Henry},
  title        = {{$PSL(2;\mathbb C)$} connections on {$3$}-manifolds with {$L^2$} bounds on curvature},
  journal      = {Cambridge Journal of Mathematics},
  volume       = {1},
  number       = {2},
  pages        = {239--397},
  year         = {2013},
  doi          = {10.4310/CJM.2013.v1.n2.a2},
  eprint       = {1205.0514},
  archivePrefix= {arXiv},
  primaryClass = {math.DG}}

@book{Taubes-ZeroLoci,
  author       = {Taubes, Clifford Henry},
  title        = {The zero loci of {$\mathbb{Z}/2$} harmonic spinors in dimension {$2$}, {$3$} and {$4$}},
  eprint       = {1407.6206},
  archivePrefix= {arXiv},
  primaryClass = {math.DG},
    year         = {2014},
  note         = {arXiv:1407.6206}
}

@incollection{Taubes-Wu,
  author       = {Taubes, Clifford Henry and Wu, Yingying},
  title        = {Examples of singularity models for {$\mathbb{Z}/2$} harmonic {$1$}-forms and spinors in dimension three},
  booktitle    = {Proceedings of the G{\"o}kova Geometry-Topology Conferences 2018/2019},
  pages        = {37--66},
  publisher    = {International Press},
  year         = {2020},
  eprint       = {2001.00227},
  archivePrefix= {arXiv},
  primaryClass = {math.DG}
}

@misc{Chen-He,
  author       = {Chen, Jiahuang and He, Siqi},
  title        = {On the existence and rigidity of critical {$Z_2$} eigenvalues},
  year         = {2024},
  eprint       = {2404.05387},
  archivePrefix= {arXiv},
  primaryClass = {math.DG},
  note         = {arXiv:2404.05387}
}

@misc{Haydys-He-Salm,
  author       = {Haydys, Andriy and He, Siqi and Salm, Willem Adriaan},
  title        = {Deformation rigidity for {$\mathbb{Z}/2$} eigensections},
  year         = {2026},
  eprint       = {2604.17044},
  archivePrefix= {arXiv},
  primaryClass = {math.DG},
  note         = {arXiv:2604.17044}
}

\end{document}